\documentclass[12pt]{amsart}
\numberwithin{equation}{section}

\allowdisplaybreaks[4]

\usepackage{amsmath}
\usepackage{amssymb,graphicx}
\usepackage{IEEEtrantools}
\usepackage{enumerate}
\usepackage{tikz}
\usepackage{comment}
\usepackage{bm}
\usepackage{mathtools}
\usepackage{microtype}
\usepackage[margin=1.15in]{geometry}
\usepackage{xcolor}
\usepackage{hyperref}
\usepackage{fancyhdr}
\usepackage{amsmath}             
\usepackage{amssymb}           
\usepackage{amsfonts}           
\usepackage{latexsym}           
\usepackage{amsthm}              
\usepackage{bbm}
\usepackage{tikz}
\usetikzlibrary{decorations.pathreplacing}
\usetikzlibrary{angles}
\usetikzlibrary{calc}
\usepackage[justification=centering]{caption}
\usepackage{type1cm}
\usepackage{todonotes}
\usepackage{standalone}
\reversemarginpar

\makeatletter
\newcommand{\labelsymbol}[2]{%
	\phantomsection% 让 hyperref 能正确跳转
	\edef\@currentlabel{#2}% 修改当前标签内容
	\label{#1}% 定义标签
	#2% 直接打印符号
}
\makeatother

\makeatletter
\newcommand\bigDiamond{\mathop{\mathpalette\bigDi@mond\relax}}
\newcommand\bigDi@mond[2]{%
	\vcenter{\hbox{\m@th
			\scalebox{\ifx#1\displaystyle 2\else1.2\fi}{$#1\Diamond$}%
	}}%
}
\newcommand\bigLozenge{\mathop{\mathpalette\bigL@zenge\relax}}
\newcommand\bigL@zenge[2]{%
	\vcenter{\hbox{\m@th
			\scalebox{\ifx#1\displaystyle 2\else1.2\fi}{$#1\blacklozenge$}%
	}}%
}
\makeatother

\makeatletter

\newcommand{\leftsymbollabel}[2]{%
	\def\@currentlabel{$#2$}% 让 \label 引用符号而非数字
	\makebox[0pt][r]{\ensuremath{#2}\, }% 在公式左边贴符号
	\label{#1}%
}
\makeatother

\hypersetup{
	colorlinks=true,
	linkcolor=black,
	anchorcolor=black,
	citecolor=black,
	filecolor=black,      
	menucolor=red,
	runcolor=black,
	urlcolor=black,
}

\theoremstyle{plain}
\newtheorem{theorem}{Theorem}[section]

\newtheorem{corollary}[theorem]{Corollary}
\newtheorem{proposition}[theorem]{Proposition}
\newtheorem*{claim}{Claim}
\newtheorem{lemma}[theorem]{Lemma}
\newtheorem{example}{Example}[section]

\theoremstyle{remark}
\newtheorem{remark}[theorem]{Remark}

\theoremstyle{definition}
\newtheorem{definition}[theorem]{Definition}

\newtheorem*{question*}{Question}

\newcommand{\ok}{\overline{\mathcal{K}}}
\newcommand{\kk}{\mathcal{K}}
\newcommand{\uk}{\underline{\mathcal{K}}}
\newcommand{\om}{\overline{M}}
\newcommand{\um}{\underline{M}}

\newcommand{\R}{\mathbb{R}}

\newcommand{\N}{\mathbb{N}}

\newcommand{\iii}{\mathbf{i}}
\newcommand{\jjj}{\mathbf{j}}
\newcommand{\bu}{\mathbf{u}}
\newcommand{\bv}{\mathbf{v}}

\newcommand{\ccc}{\mathbf{c}}
\newcommand{\rrr}{\mathtt{r}}
\newcommand{\eps}{\varepsilon}

\newcommand{\de}{\delta}
\newcommand{\f}{\frac}
\newcommand{\s}{\sigma}
\newcommand{\oc}{\overline{c}}
\newcommand{\uc}{\underline{c}}
\newcommand{\lam}{\lambda}

\renewcommand{\ge}{\geqslant}
\renewcommand{\le}{\leqslant}
\renewcommand{\geq}{\geqslant}
\renewcommand{\leq}{\leqslant}

\DeclareMathOperator{\dimh}{dim_H}

\DeclareMathOperator{\odimb}{\overline{dim}_B}

\DeclareMathOperator{\dima}{dim_A}
\DeclareMathOperator{\dimqa}{dim_{qA}}

\newcommand{\blue}[1]{{\color{blue}#1}}

\allowdisplaybreaks[4]

\title{Assouad   and quasi-Assouad dimensions of Moran sets}

\author{Jun-Jie Miao}
\address{School of Mathematical Sciences,  Key Laboratory of MEA(Ministry of Education) \& Shanghai Key Laboratory of PMMP,  East China Normal University, Shanghai 200241, China}

\email{jjmiao@math.ecnu.edu.cn}

\author{Ming-Hui Xu}
\address{School of Mathematical Sciences, East China Normal University, No. 500, Dongchuan Road, Shanghai 200241, P. R. China}
\email{52290155010@stu.ecnu.edu.cn}

\begin{document}
	\maketitle

	\begin{abstract}
		Moran sets are a non-autonomous generalization of self-similar sets.  In this paper, we study the quasi-Assouad and Assouad dimensions of Moran sets in $\mathbb{R}^{d}$.
		First we provide quasi-Assouad dimension formulae for Moran sets satisfying $c_*>0$. Then, we provide the upper and lower bounds for 
		quasi-Assouad dimension formulae for Moran sets without assuming  $c_*>0$. To obtain the exact dimension formulae in this case, we define quasi-normal
		and normal Moran sets, and provide  quasi-Assouad dimension formulae for these sets.
	\end{abstract}

	\section{Introduction}
	\subsection{Assouad and quasi-Assouad dimensions}
	Let $(X,d)$ be a metric space. For $A\subset X$ and $r>0$, denote by $N_r(A)$ the minimal number of balls of radius $r$ required to cover $A$. The \emph{Assouad dimension} of a non-empty subset $F$ of $X$ is defined by
	\begin{eqnarray*}
		\quad \dima F=\inf\Big\{\alpha:\text{there exist \ } C>0\text{\ and\ }0<\rho<1 \text{\ such that for all\ } x\in F \textit{ and }\\
		0<r<R<\rho,\vspace{0.5em}\
		N_r(B(x,R)\cap F)\le C\Big(\frac Rr\Big)^{\alpha}
		\Big\}.
	\end{eqnarray*}
We refer readers to \cite{falconer2013fractal,fraser2020assouad} for background.	The Assouad dimension was originally introduced in the study of bi-Lipschitz embeddings of metric spaces \cite{assouad1977espaces}. It has since become a fundamental notion in metric geometry and fractal analysis, capturing the most extreme local scaling behavior of sets through the maximal growth rate of covering numbers at small scales.
	This property makes the Assouad dimension an important tool in areas such as embedding theory, geometric measure theory, and the analysis of dynamics on fractals, see \cite{fraser2020assouad, heinonen2001lectures, %luukkainen1998assouad,
		robinson2010dimensions}. 

Recently, Wang and Zahl \cite{WS25} proved that in $\mathbb{R}^{3}$, every Kakeya set has Assouad dimension $3$ and every Ahlfors-David regular Kakeya
set has Hausdorff dimension $3$. In \cite{Fraser14}, Fraser systematically investigated the basic properties of Assouad-type dimensions and computed the Assouad dimensions for various classes of self-affine and quasi-self-similar sets. In \cite{BF}, Banaji and Fraser studied  the Assouad type dimensions of limit sets of iterated function systems consisting of a countably infinite number of conformal contractions, and they obtained the Assouad dimension formula for such fractals under some separation conditions.  We refer readers to \cite{BKR,fraser2018assouad,FK23} for various studies on this topic.

	However, the Assouad dimension is highly sensitive to local irregularities, which often renders it less suitable for capturing the typical geometric structure of many fractal sets. To overcome this limitation, Lü and Xi \cite{lu2016quasi} introduced the quasi-Assouad dimension, a more refined variant of the Assouad dimension. For $\eta\in (0,1)$, define
	\begin{eqnarray*}
		h_{F}(\eta)\!=\inf\Big\{\alpha:\text{there exist \ } C>0\text{\ and\ }0<\rho<1 \text{\ such that}
		\text{ for all\ } x\in F\text{\ and\ } \nonumber \\
		\quad 0<r<R^{1+\eta}<R<\rho,\  
		N_r(B(x,R)\cap F)\le C\Big(\frac Rr\Big)^{\alpha}
		\Big\},                      \label{def_hED}
	\end{eqnarray*} 
	and  the \emph{quasi-Assouad dimension}  of $F$ is defined as 
	\[
	\dimqa F=\lim_{\eta\to0}h_F(\eta).
	\]
	It is worth noting that $h_F(\eta)$ corresponds to the upper Assouad spectrum ${\overline{\dim}}^{\theta}_{\mathrm A}F$ in \cite{fraser2018assouad}, where $\theta=\frac1{\eta+1}$.
	The quasi-Assouad dimension provides a finer description of the geometric structure at intermediate scales while preserving key metric properties. For a compact subset $E$, the following chain of inequalities holds:
	\begin{equation} \label{dimHBAQ}
		\dimh E\leq \odimb E\le \dimqa E\le \dima E,
	\end{equation}
	where $\dimh$ and  $\odimb $ denote the Hausdorff and upper box dimensions, respectively. We refer readers to \cite{fraser2020assouad} for further reading. In \cite{fraser2015assouad}, Fraser proved that all dimensions in \eqref{dimHBAQ} are equal for self-similar sets satisfying the weak separation condition. The inequalities in \eqref{dimHBAQ} can strictly hold; see Example \ref{exa'}.

However, the study of the quasi-Assouad dimension in this setting remains relatively underdeveloped. Although some results for quasi-Assouad dimension have been established in \cite{lu2016quasi,garcia2021}, explicit formulas and sharp bounds for Moran sets—especially in the general (possibly non-homogeneous) case—are still largely unexplored.

	\subsection{Fractal dimensions of Moran sets }\label{cond}			
As a generalization of self-similar and Cantor-like constructions, Moran structures are distinguished by their ability to accommodate significant inhomogeneity. This inherent flexibility provides a rich framework, establishing them as a fundamental and widely applicable class of fractals. In \cite{Shmerkin}, Shmerkin  studied Bernoulli convolutions defined on a class of homogeneous Moran sets and obtained their $L^q$ dimensions under certain conditions. As an application, he settled Furstenberg's long-standing conjecture on the dimension of intersections of $\times p$ - and $\times q$-invariant sets.  We refer readers to \cite{Multi,hua2000structures,local,MW26,GM1,GM2} for further studies on Moran structures.

	Let $\{n_{k}\}_{k\geq 1}$ be a sequence of integers with $n_k\geq 2.$ For $k_1\leq k_2\in \N^+$, we write
	\begin{gather*}
		D_{k_1,k_2}=\{\bu=u_{k_1}u_{k_1+1}\cdots u_{k_2}:1\leq u_{j}\leq
		n_{j}\ , k_1\le j\le k_2\}.
	\end{gather*}
	For simplicity, we write $	D_{k}= D_{1,k}$   with $
	D_{0}=\{\emptyset \}$ containing only the empty word $\emptyset$, and write
	$$
	D^{*}=\bigcup _{k=0}^{\infty }D_{k}
	$$
	for the set of all finite words. For $\bu \in D_{k_1,k_2}$, we denote by $\bu^-$ the word obtained from $\bu$ by removing its last symbol. For $\iii=i_1\ldots i_k \in D_k$ and $\jjj=j_1\ldots j_l\in D_{k+1,k+l}$, we denote by $\iii*\jjj=i_1\ldots i_kj_1\ldots j_l\in D_{k+l}$ the concatenation of $\iii$ and $\jjj$.
 	
	Let $
	J\subset \mathbb{R}^{d}$ be a compact set with $\mbox{int}(J)\neq
	\varnothing $, where int($\cdot )$ denotes the interior of a set. Let $
	\{\ccc_{k}\}_{k\geq 1}$ be a sequence of positive real vectors where $\ccc
	_{k}=(c_{k,1},c_{k,2},\cdots ,c_{k,n_{k}})$ and $\sum
	_{j=1}^{n_{k}}c_{k,j}^{\,d}\leq 1$ for each $k\in \mathbb{N}$. We say the
	collection $\mathcal{F}=\{J_{\bu}:\bu\in D^*\}$ of closed
	subsets of $J$ fulfills the \textit{Moran structure} if it satisfies the
	following Moran structure conditions (MSC):
	\begin{enumerate}
		
		\item For each $\bu\in D^*$, $J_{\bu}$ is geometrically similar to
		$J$, i.e., there exists a similarity $\Psi_{\bu}:\mathbb{R}%
		^{d}\rightarrow \mathbb{R}^{d}$ such that $J_{\bu}=\Psi_{\bu}(J)$%
		. We write $J_{\emptyset }=J$ for empty word $\emptyset $.
		\item For all $k\in \mathbb{N}$ and $\bu\in D_{k-1}$, the elements $J_{%
			\bu*1}, J_{\bu*2},\cdots ,J_{\bu*n_{k}}$ of $\mathcal{F}$
		are the subsets of $J_{\bu}$ with   $\mbox{int}%
		(J_{\bu*i})\cap \mbox{int}(J_{\bu*i^{\prime }})=\varnothing $
		for $i\neq i^{\prime }$ and 
		\begin{equation*}
			\frac{|J_{\bu*i}|}{|J_{\bu}|}=c_{k,i},
		\end{equation*}%
		for all $1\leq i\leq n_{k} $,		where $|\cdot |$ denotes the diameter.
		
		%For each $\bu\in D^*$, there is $x\in J_{\bu}$ such that $B(x,C|J_{\bu}|)\subset J_{\bu}$ ($C$ is a uniform constant.).        
	\end{enumerate}
	The non-empty compact set
	\begin{equation}\label{attractor}
		%E=E(\mathcal{F})=\bigcap\nolimits_{k=1}^{\infty }\bigcup\nolimits_{%\sigma\in \Omega^{k}}J_{\sigma}
		E=E(\mathcal{F})=\bigcap_{k=1}^{\infty }\,\bigcup_{%
			\bu\in D_{k}}J_{\bu}
	\end{equation}
	is called a \textit{Moran set} determined by $\mathcal{F}$ or $(J,\{n_k\},\{\ccc_k\})$. We write $\mathcal{M}(J,\{n_k\},\{\ccc_k\})$ for the collection of all sets determined by $\mathcal{F}$. For all $\bu\in D_{k}$, the elements $J_{\bu}$
	are called \textit{\ $k$th-level basic sets} of $E$.   We say that $\mathcal{M}(J,\{n_k\},\{\ccc_k\})$ or $E$ is \emph{homogeneous} if for each $k\in \N$, there is a number $c_k$ such that $c_{k,j}=c_k, 1\le j\le n_k$. We refer readers to \cite{Wen2000} for the background reading.
	
	In the article, we always assume  $J=[0,1]^d$. For $p,q\in \N^+$ with $1\le p< q$ and $\bu=u_p\ldots u_{q}\in D_{p,q}$,  we denote $c_\bu=c_{p,u_p}c_{p+1,u_{p+1}}\cdots c_{q,u_{q}}$. We also write
	\[ \begin{split}
		\om_{p,q}=\max_{\bu\in D_{p,q}}c_{\bu},\quad & \quad \um_{p,q}=\min_{\bu\in D_{p,q}}c_{\bu}   \\ \oc_k=\max_{1\le j\le n_k}c_{k,j},\quad &\quad \uc_k=\min_{1\le j\le n_k}c_{k,j}, 
	\end{split}	\]
	and $
	\om_k =\om_{1,k}=\max\limits_{\bu \in D_k}c_{\bu}$, $  \um_k = \um_{1,k}=\min\limits_{\bu \in D_k}c_{\bu} $ for simplicity.	
	For $k\in \N^+,l\in \N $, let $s_{k,k+l}$  be the unique real solution of the equation
	\begin{equation} \label{def_skl}
		\prod_{i=k}^{k+l}\Big(\sum_{j=1}^{n_i}c_{i,j}^s \Big)=1.
	\end{equation}
	and write $s_k = s_{k,k}$ for simplicity.  

The fractal dimensions of  Moran sets are usually determined by the double index sequence $\{s_{k,k+l}\}$.   If  $\mathcal{M}(J,\{n_k\},\{\ccc_k\})$ satisfies  $c_*=\inf_{k,j}c_{k,j}>0$,  then, for any  $E\in\mathcal{M}(J,\{n_k\},\{\ccc_k\})$,
	\begin{equation}\label{Eq_HBAdim}
		\dimh E=\liminf_{l\to\infty} s_{1,l},\quad \odimb E=\limsup_{l\to\infty} s_{1,l}, \quad  \dima E=\limsup_{l\to\infty}\sup_{k\ge 1} s_{k+1,k+l} .
	\end{equation}
	See  \cite{li2016assouad} for details.  
	In \cite{hua2000structures}, Hua, Rao, wen and Wu proved that   the Hausdorff  and upper box dimension formulae in \eqref{Eq_HBAdim} still hold if $c_*>0$ is relaxed to the following condition:
	\begin{equation}\label{club}
		\lim_{k\to\infty}\frac{\log \uc_k}{\log \om_k}=0.
	\end{equation}
Conditions $c_*>0$ and \eqref{club} are most frequently used in the study of Moran-type sets. 

	Given $\iii\in D^*$ and $\de>0$, we define
	\[
	D_{\iii}(\de)=\{\iii*\jjj\in D^*: c_{\jjj}\le \de < c_{\jjj^-}\}.
	\]
	In particular, when $\iii=\emptyset$, we write $D(\de)$ instead of $D_{\emptyset}(\de)$; that is,
	\[
	D(\de)=\{\jjj\in D^*: c_{\jjj}\le \de < c_{\jjj^-}\}.
	\]
	Recently,  K{\"a}enm{\"a}ki and Rutar \cite{kaenmaki2024regularity}  showed that the Assouad dimension in \eqref{Eq_HBAdim}  remains valid if $c_*>0$ is replaced by the bounded neighbourhood condition which  bears a strong relation to the finite clustering property studied in \cite{kaenmaki2015weak,MR2415480,rajala2011weakly}. 
	\begin{definition} 
		Given  $E\in \mathcal{M}(J,\{n_k\},\{\ccc_k\})$, we say that $E$ satisfies the \textit{bounded neighbourhood condition} (BNC) if
		\[
		\limsup_{\de\to 0}\sup_{x\in E}\# \left\{\bu\in D(\de):E\cap J_{\bu}\cap B(x,\de)\neq \emptyset \right\}<\infty.
		\]
	\end{definition}
In Lemma \ref{numberofcap}, we show that \eqref{club} implies a weak version of BNC.  
In  \cite{kaenmaki2024regularity}, 	K{\"a}enm{\"a}ki and Rutar also introduced the following.
	\begin{definition}
		Given  $\mathcal{M}(J,\{n_k\},\{\ccc_k\})$, we say that $\mathcal{M}(J,\{n_k\},\{\ccc_k\})$ satisfied the  {\em bounded branching condition} (BBC) if
		\begin{equation*}
			\limsup_{\de\to 0}\sup_{\bu\in D(\de)}\# \left\{\bu'\in D(\de):\bu'=\bu^-*\iii\text {\ for some }\iii\, \right\}<\infty.
		\end{equation*}
	\end{definition}

We prove that BBC is sufficient for the BNC (Proposition \ref{BC->BNC}), which partially extends \cite[Theorem B]{kaenmaki2024regularity}. An equivalent characterization of BBC is provided in Proposition \ref{BC}. Observe that the implication $c_*>0 \Rightarrow (\text{\eqref{club}} \text{ and } \text{BBC})$ is straightforward (Corollary  \ref{c_*->BC}).

 Note that conditions \eqref{club} and  BBC  are mutually independent. They control different aspects of the Moran structure. Specifically, \eqref{club} bounds the relative scaling between the ``thinnest" and ``fattest" basic sets across levels, and the BBC \cite{kaenmaki2024regularity} restricts the number of ``geometric siblings" in $D(\delta)$.  The essential difference between the BNC and the two aforementioned conditions lies in that the BNC  depends strongly on the locations of the sets in $\mathcal{F}=\{J_{\bu}:\bu\in D^*\}$, whereas  \eqref{club} and  BBC do not.
Examples illustrating these cases are presented below, with proofs given in Section \ref{Sec_ADE}.
	
	\begin{example}\label{1-2-3+}
		 $(1)$  The Moran structure $ \mathcal{M}\big(J,\{n_k\equiv 2\},\big\{(2^{-2^k},2^{-2^k})\big\}_{k\ge 1}\big)$ satisfies BBC but fails to satisfy \eqref{club}.\vspace{0.2em}
		
\noindent		(2)  The Moran structure $\mathcal{M}\big(J,\{k+1\}_{k\ge 1},\big\{\big(\f1{2k},\ldots,\f1{2k}\big)\big\}_{k\ge 1}\big)$ satisfies \eqref{club} but fails to satisfy BBC.
		%(2)	For $ \mathcal{M}\Big(J,\{2,2,\dots\},\big\{\big(\f12,2^{-k}\big)\big\}_{k\ge 1}\Big)$, it  satisfies \ref{itm:index}, yet it does not satisfy BBC and \eqref{club}.%does not satisfy  either  \eqref{club} or  BNC.
	\end{example}	
	The following example was suggested by Alex Rutar in a private conversation,  illustrating  that the BNC does not necessarily imply BBC.
	\begin{example}\label{BCandBNC}
		Let $E\in \mathcal{M}([0,1],\{n_k\equiv 2\},\{\ccc_{k}\})$, where
		\[
		c_{k,1}=\left\{\begin{array}{cl}
			\f 1{1+2^n},\, & k=\f{n(n+1)}2 \vspace{0.3em}\\
			\f12, & \text{otherwise}\\
		\end{array},
		\right.\ c_{k,2}=\left\{\begin{array}{cl}
			\f {2^n}{1+2^n},\, & k=\f{n(n+1)}2 \vspace{0.3em}\\
			\f12, & \text{otherwise}\\
		\end{array},
		\right.\ n\in\N^+.
		\]
Assume that for each $\iii\in D^*$, the left and right endpoints of 
		$J_\iii$ coincide  respectively  with the left endpoint of
		$J_{\iii*1}$ and the right endpoint of $J_{\iii*2}$.	Then $E=[0,1]$ satisfies the BNC, but BBC does not hold.
	\end{example}

	In \cite{PENG2017192,yang2020assouad}, the authors  provide a different Assouad dimension   formula for homogeneous Moran  sets $E\in \mathcal{J}(I_0,\{n_k\},\{c_k\},\{\eta_{k,j}\})$ with $\sup_k n_k<\infty$, 
	\[
	\dima E=\lim_{m\to\infty}\sup_{k\ge m,l\ge m} s_{k+1,k+l}.
	\]
	\iffalse
	\begin{theorem}\cite{yang2020assouad}
		Let $E\in \mathcal{J}(I_0,\{n_k\},\{c_k\},\{\eta_{k,j}\})$ be a homogeneous perfect set. If $\sup_k n_k<\infty$, then 
		\[
		\dima E=\lim_{m\to\infty}\sup_{k\ge m,l\ge m} s_{k+1,k+l}=\lim_{m\to\infty}\sup_{k\ge m,l\ge m}\f{\log n_{k+1}\cdots n_{k+l}}{-\log c_{k+1}\cdots c_{k+l}}.
		\]
	\end{theorem}
	\fi
	Consequently, it is natural to ask under what conditions the two formulae are equivalent, that is, 
	\[
	\lim_{m\to\infty}\sup_{k\ge m,l\ge m} s_{k+1,k+l}=\limsup_{l\to\infty}\sup_{k\ge 1} s_{k+1,k+l}.
	\]
A sufficient condition for this equality is derived as a direct consequence of our principal conclusions; see Propostion \ref{general} and  Section \ref{Sec_ADE} for details.

	In the paper, we study the Assouad and quasi-Assouad dimensions of Moran sets in $\mathbb{R}^{d}$ and establish our main findings in the subsequent section.

	\section{Notation and main results}
	Given   $\mathcal{M}(J,\{n_k\},\{\ccc_k\})$, we  write
	\begin{equation}\label{def_t*tt*}
		\begin{split}
			&t^*=\lim_{\eta\to 0}\limsup_{l\to\infty} \sup_{k\in \ok_{l+1,\eta}} s_{k+1,k+l}, \qquad
			t=\lim_{\eta\to 0}\limsup_{l\to\infty}\sup_{k\in \kk_{l,\eta}} s_{k+1,k+l}, \\
			&t_*=\lim_{\eta\to 0}\limsup_{l\to\infty}\sup_{k\in \uk_{l,\eta}} s_{k+1,k+l}, 
		\end{split}
	\end{equation}
	where
	\begin{eqnarray*}
		&&\ok_{l,\eta}=\Big\{k\ge 2:\frac{\log \um_{k+1,k+l}}{\log \om_{k-1}}>\eta\Big\}, \qquad \kk_{l,\eta}=\Big\{k:\frac{\log \um_{k+1,k+l}}{\log \om_{k}}>\eta\Big\},\\
		&&\uk_{l,\eta}=\Big\{k:\frac{\log \om_{k+1,k+l}}{\log \om_{k}}>\eta\Big\}. 
	\end{eqnarray*}
Then, for $0<\eta<\f{\log \oc_2}{\log \oc_1}$, $1\in\uk_{l,\eta}\subseteq\kk_{l,\eta}\subseteq\ok_{l,\eta}\subseteq\ok_{l+1,\eta}$ for all $l\geq 1$ and thus
	\begin{align*}
		\limsup_{l\to\infty}s_{1,l}\le t_*\le  t\le t^*\le \limsup_{l\to\infty}\sup_{k\ge1} s_{k+1,k+l}.
	\end{align*}
Note that for each $\eta>0$, it is clear that  $	\bigcup_{l\ge 1}\,\kk_{l,\eta}=\bigcup_{l\ge 1}\,\ok_{l,\eta}=\N^+,$ and 
		\[
		\ok_{l,\eta}\subseteq\ok_{l+1,\eta},\qquad\kk_{l,\eta}\subseteq\kk_{l+1,\eta},\qquad \uk_{l,\eta}\subseteq\uk_{l+1,\eta} . 
		\]
		Moreover, if $\lim_{k\to\infty} \om_k=0$, then 
		$ \bigcup_{l\ge 1}\,\uk_{l,\eta}=\N^+. $

	For $c_*>0$, we obtain  the quasi-Assouad dimension of Moran sets.
	\begin{theorem}\label{no.1}
If $\mathcal{M}(J,\{n_k\},\{\ccc_k\})$  has $c_*>0$, then every $E\in\mathcal{M}(J,\{n_k\},\{\ccc_k\})$ satisfies
		\[
		\dimqa E=t_*=  t= t^*.
		\]
	\end{theorem}

	To derive  dimension formulae  without  $c_*>0$,  we  first provide two special results.

	\begin{proposition}\label{general}
If $\mathcal{M}(J,\{n_k\},\{\ccc_k\})$ has  $\lim_{k\to\infty}\om_k=0$, then  for all $K\in \N^+$,
		\[
		\limsup_{l\to\infty}\sup_{k\ge K} s_{k,k+l}=\lim_{m\to \infty}\sup_{k\ge m,l\ge m}s_{k,k+l}.
		\]
	\end{proposition}

Proposition \ref{general} serves to unify the Assouad dimension formulas presented in \cite{PENG2017192,yang2020assouad} and \cite{li2016assouad}. The assumption $\lim_{k\to\infty} \om_k = 0$ in Proposition \ref{general} is crucial, as the proposition fails when the limit is positive; see Example \ref{exa} (i).	

\begin{theorem}\label{WhenBNCisneeded}
If $\mathcal{M}(J,\{n_k\},\{\ccc_k\})$ satisfies $\liminf_{k\to\infty}\f{\log \uc_k}{\log \om_k}>0$ and BBC, then every  $E\in\mathcal{M}(J,\{n_k\},\{\ccc_k\})$ satisfies
		\[
		\dimqa E=\dima E=\limsup_{l\to\infty}\sup_{k\ge 1} s_{k+1,k+l}.
		\]
	\end{theorem}

Theorem \ref{WhenBNCisneeded} in fact identifies a class of Moran sets that are quasi-Lipschitz Assouad-minimal (i.e., compact sets whose Assouad dimension does not decrease under any quasi-Lipschitz map). The detailed arguments can be found in \cite{lu2016quasi}.

To investigate the quasi-Assouad dimensions of Moran sets without assuming $c_*>0$, we introduce the following definitions.
	\begin{definition}
The Moran structure $\mathcal{M}(J,\{n_k\},\{\ccc_k\})$ is said to be {\em quasi-normal } if there exists a strictly increasing function $\varphi:(0,+\infty)\to (0,+\infty)$ such that
		\[
		\limsup_{l\to \infty}\sup_{k\ge 1}\, \varphi\left(\frac{\log \um_{k+1,k+l}}{\log \om_{k}}\right)\bigg/\frac{\log \om_{k+1,k+l}}{\log \om_{k}}<\infty.
		\]
		It is said to be {\em normal} if there also exists a strictly increasing function $\varphi:(0,+\infty)\to (0,+\infty)$ such that
		 \[
		 \limsup_{l\to \infty}\sup_{k\ge 2}\, \varphi\left(\frac{\log \um_{k+1,k+l+1}}{\log \om_{k-1}}\right)\bigg/\frac{\log \om_{k+1,k+l}}{\log \om_{k}}<\infty.\]

	\end{definition}

The generality of this definition is evidenced by two key cases: first, any $ \mathcal{M}(J,\{n_k\},\{\ccc_k\})$ satisfying $c_*>0$ is normal and consequently quasi-normal (Proposition \ref{quasihomo}); second, all homogeneous Moran structures are quasi-normal, dispensing with the $c_*>0$ requirement altogether.  As direct verification of these properties is often challenging, Proposition \ref{quasihomo} offers a set of simpler and more verifiable criteria.
	
	Next,  we state our main results on quasi-Assouad dimensions of quasi-normal and normal Moran sets.
	\begin{theorem}\label{main1}
		Let $\mathcal{M}(J,\{n_k\},\{\ccc_k\})$ be quasi-normal and satisfy \eqref{club}. Then every   $E\in\mathcal{M}(J,\{n_k\},\{\ccc_k\})$ satisfies
		\[
		\dimqa E=t_*=t.
		\]
	\end{theorem}
	
		\begin{corollary}\label{t_*=t^*}
		Let $\mathcal{M}(J,\{n_k\},\{\ccc_k\})$ be quasi-normal and satisfy
\begin{equation} \label{cdnllZ}
		\lim_{l\to \infty}\sup_{k\ge 1}\f{\log \uc_k}{\log \om_{k+1,k+l}}=\lim_{l\to \infty}\sup_{k\ge 1}\f{\log \uc_{k+l+1}}{\log \om_{k+1,k+l}}=0.
\end{equation}
		 Then every $E\in\mathcal{M}(J,\{n_k\},\{\ccc_k\})$ satisfies
		\begin{equation*}
			\dimqa E=t_*=t=t^*.
		\end{equation*}
	\end{corollary}

	The next  conclusion is a direct consequence of Theorem \ref{main1} which was first proved by   L\"{u} and Xi in \cite{lu2016quasi}, 
	
	\begin{corollary} 
		Let $\mathcal{M}(J,\{n_k\},\{\ccc_k\})$ be homogeneous with   \eqref{club}. Then every $E\in\mathcal{M}(J,\{n_k\},\{\ccc_k\})$ satisfies
		\begin{equation*}
			\dimqa E=t_*=t.
		\end{equation*}
	\end{corollary}
	
The final result concerns the quasi-Assouad dimensions of normal Moran sets.
	\begin{theorem}\label{main2}
		Let $\mathcal{M}(J,\{n_k\},\{\ccc_k\})$ be normal and satisfy BBC. Then every $E\in\mathcal{M}(J,\{n_k\},\{\ccc_k\})$ satisfies
		\[
		\dimqa E=t_*=t^*.
		\]
	\end{theorem}

Finally, we conclude by presenting two examples. The first example shows that the dimension inequality \eqref{dimHBAQ} is typically strict for Moran fractals. The second example demonstrates that the conclusions of Theorems \ref{main1} and \ref{main2} may fail when either \eqref{club} or BBC does not hold, and that
	$\lim_{k\to\infty} \om_k>0$ does not necessarily imply that the sequence $\{s_{k,k+l}\}$ exhibits ``trivial" behavior.  	
	\begin{example}\label{exa'}
		Let ${a_n}$ be a sequence of positive integers with $a_1=1$, satisfying $a_{n+1}>2a_n+n$ for all $n$, and assume that
		\[
		\lim_{n\to\infty}\f n{a_n}=\lim_{n\to\infty}\f {a_1+a_2+\cdots+a_{n-1}}{a_n}=0.
		\]
		Let $E\in \mathcal{M}([0,1],\{n_k\},\{\ccc_{k}\})$, where
		\[
		n_k\equiv 2,\quad c_{k,1}=\f12,\quad c_{k,2}=\left\{\begin{array}{ll}
			\f 14,\quad & a_n\le k<2a_n\vspace{0.5em}\\
			\f12, & 2a_n\le k<2a_n+n\vspace{0.5em}\\
			\f18, & 2a_n+n\le k<a_{n+1}
		\end{array}
		\right.,\quad n\in\N^+.
		\]
		Then the structure is normal, and the dimensions of $E$ satisfy
		\[
		\dimh E=-\f{\log p}{\log 2}<\odimb E=-\f{\log q}{\log 2}<\dimqa E=-\f{\log r}{\log 2}<\dima E=1,
		\]
		where $p,q,r$ satisfy $p+p^3=(q+q^2)(q+q^3)=r+r^2=1.$
	\end{example}
	
	%$\lim_{k\to\infty} \om_k=0$ is not a  reasonable  assumption for our purpose. 
	\begin{example}\label{exa}
		Given $0<\alpha<\beta$, let $E\in \mathcal{M}([0,1],\{n_k\},\{\ccc_{k}\})$ where
		\[
		n_k\equiv 2,\quad c_{k,1}=1- (k+1)^{-\alpha},\quad c_{k,2}=(k+1)^{-\beta}.
		\] 
		Then the structure is not quasi-normal, and
		\[
		\lim_{m\to\infty}\sup_{k\ge m,l\ge m}s_{k,k+l}=\f\alpha\beta.
		\]
Assume that for any $\iii\in D^*$, the left and right endpoints of 
		$J_\iii$ coincide  respectively  with the left endpoint of
		$J_{\iii*1}$ and the right endpoint of $J_{\iii*2}$. Then
		\begin{enumerate}[(i)]
			\item For $\alpha>1$, one has $\dimqa E=\dima E=1$, whereas
			\[
			\f\alpha\beta<\lim_{l\to\infty} s_{1,l}\le\lim_{l\to\infty}\sup_{k\ge 1} s_{k,k+l}\le\max \Big\{\f\alpha\beta+\f{\log\beta-\log\alpha}{\beta\log 3},s_{1}\Big\}<1.
			\]
			In particular, if $\beta\ge\f{2^\alpha+1}{2^\alpha-1}2^{\f{\alpha}{2^\alpha-1}}\alpha$\footnote[2]{Note that $\f{2^\alpha+1}{2^\alpha-1}2^{\f{\alpha}{2^\alpha-1}}$ converges to $1$ exponentially fast as $\alpha$ tends to $\infty$. Hence, roughly speaking, this means that in ``most" cases with $\beta>\alpha>1$, we have $\lim_{l\to\infty} s_{1,l}=\lim_{l\to\infty}\sup_{k\ge 1} s_{k,k+l}$.}, then $\lim_{l\to\infty} s_{1,l}=\lim_{l\to\infty}\sup_{k\ge1} s_{k,k+l}.$ 
			\item For $\alpha= 1$, we have 
			\[
			\f\alpha\beta=\lim_{l\to\infty} s_{1,l}=\lim_{l\to\infty}\sup_{k\ge1} s_{k,k+l}<\dimqa E=\dima E=1.
			\]
			
			\item For $0<\alpha< 1$, we have 
			\[
			\f\alpha\beta=\lim_{l\to\infty} s_{1,l}=\lim_{l\to\infty}\sup_{k\ge1} s_{k,k+l}=\dimqa E<\dima E=1.
			\]
		\end{enumerate} 
	\end{example}

	\iffalse	
	\blue{	
		\begin{remark}
			Observe that $\f{2^\alpha+1}{2^\alpha-1}2^{\f{\alpha}{2^\alpha-1}}$ approaches 1 exponentially fast. Hence, in ``most" cases with $\beta>\alpha>1$, we have $\lim_{l\to\infty} s_{1,l}=\lim_{l\to\infty}\sup_k s_{k,k+l}$.
		\end{remark}
	}
	\fi
	The remainder of the paper  is organized as follows. In Section \ref{sec_qhs}, we study quasi-normal and normal sets.
	In Section \ref{sec_glb}, we show that $t_*$ is  a general lower bound for the quasi-Assouad dimension of Moran sets. In Section \ref{sec_qadms}, we establish the quasi-Assouad dimension formulae for Moran sets under the assumptions of \eqref{club} or BBC, and provide the proofs of Theorems \ref{no.1}, \ref{main1}, and \ref{main2}. Finally, we investigate the dimension formulae of Moran sets and prove Theorem \ref{WhenBNCisneeded} and Proposition \ref{general}; the detailed computations for the examples discussed in this paper are presented in Section \ref{Sec_ADE}.

	\section{ Quasi-normal sets }\label{sec_qhs}
The properties of BBC, quasi-normality, and normality for a Moran structure are inherently difficult to verify directly from their definitions. To remedy this, we establish several alternative characterizations that are substantially easier to check in this section.

	\begin{proposition}\label{BC}
		For the Moran structure $\mathcal{M}(J,\{n_k\},\{\ccc_k\})$, define  $b_k=\min\big\{l:\om_{k,k+l}\le \uc_k\big\}$. Then $\mathcal{M}(J,\{n_k\},\{\ccc_k\})$ satisfies BBC  if and only if $	\sup_{k\ge 1} b_k <\infty \text{\ and } \sup_{k\ge 1} n_k<\infty.$	\end{proposition}

	\begin{proof}
		Without loss of generality, assume that for each $k\geq 1$, 
		\[
		c_{k,1}\ge c_{k,2}\ge\cdots\ge c_{k,n_k}.
		\]
		Denote by $\bu_k=11\ldots1\in D_k$.

		Suppose that BBC holds, and let $\de_k=\om_{k-1}\uc_k$ for $k\ge2$.  Since $\uc_k^d\le 1-\oc_k^d$ for every $k>0$, $\lim_{k\to \infty} \delta_k=0$. We obtain that
		\[
		\{\bu\in D(\de_k):
		\bu=\bu_m*i*\jjj \text{\ for some\ }\jjj\,
		\}
		\neq \emptyset
		\]
		for $	k-1\le m\le k+b_k-1,\,1\le i\le n_{m+1}$. Note that the case $\jjj=\emptyset$ may occur.
		
		Therefore, it follows that
		\begin{IEEEeqnarray*}{rCl}
			&&\limsup_{k\to\infty}\sup_{\bu\in D(\de_k)}\#\{\bu'\in D(\de_k):\bu'=\bu^-*\iii\text{\ for\ some\ }\iii\}\\
			&\ge& \limsup_{k\to\infty}\#\{\bu'\in D(\de_k):\bu'= \bu_{k-1}*\iii\text{\ for\ some\ }\iii\}\\
			&\ge
			&\limsup_{k\to\infty}\Big(\sum_{i=0}^{b_k-1}\sum_{j=2}^{n_{k+i}}\#	\left\{\bu\in D(\de_k): 
			\bu=\bu_{k+i-1}*j*\jjj \text{ for some }\jjj
			\right\}  +n_{k+b_k}\Big)\\
			%\big\{\bu'\in D(\de_k):\bu'= \bu_{k+i-1}*j*\iii\in D^*\text{\ for\ some\ }\iii \big\}+n_{k+b_k}\Big)\\
			&\ge&\limsup_{k\to\infty} \Big(\sum_{i=0}^{b_k}n_{k+i}-b_k\Big).
		\end{IEEEeqnarray*}
		Moreover, the preceding argument implies that $b_k<\infty$ for all $k$. Thus we immediately obtain $\sup_{k\ge 1}\Big(\sum_{i=0}^{b_k}n_{k+i}-b_k\Big)<\infty,$ and the conclusion holds since $n_k\ge 2$ for each $k$.
		
		Conversely, suppose that $\sup_k b_k<L$ and $\sup_k n_k<N$ for some integers $L,N$. Fix  $\de>0$. Since $	c_{\bu^-*\bu_L}\le c_{\bu^-*n_{|\bu|}}\le c_\bu\le \de $  for all $\bu\in D(\de)$,  
		we have
		\[
		\sup_{\bu\in D(\de)}\{|\iii|:\bu'=\bu^-*\iii\in D(\de)\}\le L, 
		\]
and it implies that 
		\begin{align*}
\sup_{\bu\in D(\de)}\#\{\bu'\in D(\de):\bu'=\bu^-*\iii\text{ for some }\iii\}
			&\le \sup_{\bu\in D(\de)}\# D_{|\bu|,|\bu|+L-1} \le N^L.
		\end{align*}
Letting $\delta \to 0$ completes the proof.
	\end{proof}
	
	\begin{remark}
		In contrast to \eqref{club}, $\sup_{k\ge 1} b_k<\infty$ indicates that BBC provides a control between the contraction ratio of the “thinnest” basic set at each level and the cumulative contraction ratio of the “fattest” basic sets in the subsequent levels of Moran sets.
	\end{remark}

	%The following conclusion was also proved in \cite{kaenmaki2024regularity}, and  we present a simple  and different   proof for it. 
	
	\begin{corollary}\label{contr}
		Suppose $ \mathcal{M}(J,\{n_k\},\{\ccc_k\})$ satisfies BBC.
		Then
		\[
		\lim_{l\to \infty}\sup_{k\ge 1} \om_{k,k+l}=0.
		\]
	\end{corollary}
	\begin{proof}
	By Proposition \ref{BC}, there exists $L>0$ such that $\om_{k,k+L-1}<\uc_k$ for all $k$. For each $l>0$, write $l=pL+r$ for some integer $p>0$ and remainder $r\in\{0,1,\ldots,L-1\}$. Then
		\begin{equation*}
			\begin{aligned}
				\limsup_{l\to \infty}\sup_{k\ge 1} \om_{k,k+l}\le 	\limsup_{l\to \infty}\sup_{k\ge 1} \uc_k\uc_{k+L}\cdots\uc_{k+(p-1)L}\le \limsup_{l\to \infty}\f1{2^p}=0.
			\end{aligned}\qedhere
		\end{equation*}
	\end{proof}

	\begin{remark}
		It was proved in \cite{kaenmaki2024regularity} that BNC suffices to imply $\lim_{l\to \infty}\sup_{k\ge 1} \om_{k,k+l}=0$ (referred to as the locally contracting condition).
		%Corollary \ref{contr} implies that if $\limsup_{l\to\infty}\sup_k \om_{k,k+l}>0$,   the Moran set  does not have bounded branching.
	\end{remark}
	
	The following results are direct consequences of Proposition \ref{BC}.
	\begin{corollary}\label{c_*->BC}
		If $\mathcal{M}(J,\{n_k\},\{\ccc_k\})$ satisfies $c_*>0$, then it satisfies BBC.
	\end{corollary}
	\begin{corollary}
		If $\mathcal{M}(J,\{n_k\},\{\ccc_k\})$ satisfies $\sup_{k\ge 1}\min\big\{l:\om_{k,k+l}\le \uc_k\big\}<\infty$ and $\inf_{k\ge 1}\oc_k>0$, then $c_*>0$.
	\end{corollary}
	
	\iffalse
	\begin{proof}
		By Proposition \ref{BC}, $c_*\ge \big(\inf_k\oc_k\big)^{\sup_k b_k}>0$, and the conclusion  follows.
	\end{proof}
	\fi
	
Finally, we provide several sufficient and necessary conditions for the quasi-normality and normality of a Moran structure.
	
	\begin{proposition}\label{quasihomo}
For the Moran structure $\mathcal{M}(J,\{n_k\},\{\ccc_k\})$, the following statements hold. 
		\begin{enumerate}[(i)]
			\item $\mathcal{M}(J,\{n_k\},\{\ccc_k\})$ is normal if one of the following  holds:
			\begin{enumerate}[(a)]
				\item $c_*>0$. 
				\item $\liminf_{k\to\infty}\f{\log \oc_k}{\log \om_k}>0$. 
				\item $\liminf_{k\to\infty}\f{\log \uc_k}{\log \om_k}>0$ and BBC holds. 
			\end{enumerate}
			
			\item $\mathcal{M}(J,\{n_k\},\{\ccc_k\})$ is not quasi-normal if one of the following  holds:%\vspace{0.2em}
			\begin{enumerate}[(a)]
				\item $\lim_{l\to\infty}\sup_k \om_{k,k+l}=1$ and $  \ \liminf_{k\to\infty}\f{\log \uc_k}{\log \om_k}>0$.    %\vspace{0.5em}
				\item $\lim_{k\to\infty} \om_k=0,\ \limsup_{l\to\infty}\sup_k \om_{k,k+l}>0$ and $ \liminf_{k\to\infty}\f{\log \uc_k}{\log \om_k}>0$.      %\vspace{0.5em}
				\item $0<\inf_k \oc_k\le\sup_k \oc_k<1$ and $ \limsup_{k\to\infty}\f{\log \uc_k}{\log \om_k}>0$.
			\end{enumerate} 
		\end{enumerate}
	\end{proposition}
	\begin{proof}%[Proof of Proposition \ref{quasihomo}]
		(i) (a) Taking $\varphi(x)=x$,  we obtain
		\begin{align*}
			\varphi\Big(\frac{\log \um_{k+1,k+l+1}}{\log \om_{k-1}}\Big)\Big/\frac{\log \om_{k+1,k+l}}{\log \om_{k}}&=\f{\log \um_{k+1,k+l+1}}{\log \om_{k+1,k+l}}\cdot \f{\log \om_k}{\log \om_{k-1}}\\
			&\le \f{d(l+1)\log c_*}{l\log (1-c_*^d)}\cdot \f{dk\log c_*}{(k-1)\log (1-c_*^d)}\\
			&\le 4 \Big(\f{d\log c_*}{\log (1-c_*^d)}\Big)^2
		\end{align*}
		for each $k\ge 2$ and $l$. Hence, $\mathcal{M}(J,\{n_k\},\{\ccc_k\})$ is normal. 
		
		(b)  Let $\varphi:(0,\infty)\to (0,\infty)$ be a strictly increasing function such that $\varphi(x)<N$ for some $N>0$. Since $\liminf_{k\to\infty}\f{\log \oc_k}{\log \um_k}>0$,  it follows that
		\[
		\sup_{k\ge 2}\varphi\Big(\frac{\log \um_{k+1,k+l+1}}{\log \om_{k-1}}\Big)\Big/\frac{\log \om_{k+1,k+l}}{\log \om_{k}}\le N\sup_{k\ge 2}\frac{\log \om_{k+1}}{\log \oc_{k+1}}<\infty
		\]
		for all $l$.
		Thus $\mathcal{M}(J,\{n_k\},\{\ccc_k\})$ is normal.
		
		(c) By Proposition \ref{BC}, there exists an integer $L$ such that $	\om_{k,k+l}<\uc_k $	for all $k$ and $l>L$. Let $\varphi:(0,\infty)\to (0,\infty)$ be a strictly increasing function such that $\varphi(x)<N'$ for some $N'>0$. Since $\liminf_{k\to\infty}\f{\log \uc_k}{\log \om_k}>0$ for any $l>L$, it follows that
		\begin{align*}
			\sup_{k\ge 2}\varphi\Big(\frac{\log \um_{k+1,k+l+1}}{\log \om_{k-1}}\Big)\Big/\frac{\log \om_{k+1,k+l}}{\log \om_{k}}\le N'\sup_{k\ge 2}\frac{\log \om_{k+1}}{\log \uc_{k+1}}<\infty.
		\end{align*}
		Thus $\mathcal{M}(J,\{n_k\},\{\ccc_k\})$ is normal.

		(ii) Let $\varphi:(0,\infty)\to (0,\infty)$ be strictly increasing. Since $\liminf_{k\to\infty}\frac{\log \uc_k}{\log \om_k}>0$, it follows that $\inf_k \frac{\log \uc_k}{\log \om_k}>0$. Furthermore, define
        \[
		\alpha\coloneq\inf_{k,l\ge 1}\f{\log \um_{k+1,k+l}}{\log \om_k}\ge \inf_{k\ge 1}\f{\log \uc_{k+1}}{\log \om_{k+1}}>0.
		\]
	
		(a) There exist sequences $\{l_n\},\{k_n\}$ with $\lim_{n\to\infty}l_n =\infty$  such that
		\[
		\lim_{n\to\infty}\om_{k_n,k_n+l_n}=\lim_{l\to\infty}\sup_k \om_{k,k+l}=1.
		\]
		Since $\varphi$ is strictly increasing, we obtain
		\begin{align*}
			\varphi\left(\frac{\log \um_{k_n+1,k+l_n}}{\log \om_{k_n}}\right)\bigg/\frac{\log \om_{k_n+1,k_n+l_n}}{\log \om_{k_n}}&\ge \f{\log \om_{k_n}}{\log \om_{k_n,k_n+l_n}}\varphi\left(\alpha\right)
			>\f{\log \oc_1}{\log \om_{k_n,k_n+l_n}}\varphi\left(\alpha\right).  
		\end{align*}
		Thus $\mathcal{M}(J,\{n_k\},\{\ccc_k\})$ is not quasi-normal.
		
		(b) Write $\beta= \limsup_{l\to\infty}\sup_k \om_{k,k+l}>0$. 
		There exist sequences $\{l'_n\},\{k'_n\}$ with $\lim_{n\to\infty} l'_n=\infty$  such that $	\om_{k'_n,k'_n+l'_n}>\frac{\beta}2 $ for all $n>0$. 
		
		We claim that $\lim_{n\to\infty} k'_n=\infty $. Otherwise, there exist  $\{k'_{n_p}\}$  and  $K>0$ such that
		\[
		\om_{k'_{n_p},k'_{n_p}+l'_{n_p}}\le\max_{1\le k'\le K} \om_{k',k'+l'_{n_p}}\to 0\ (p\to\infty).
		\]	
which contradicts the fact $\om_{k'_n,k'_n+l'_n}>\frac{\beta}2 $ for all $n>0$. 		
		Therefore,  we obtain
		\begin{align*}
			\varphi\left(\frac{\log \um_{k'_n+1,k'_n+l'_n}}{\log \om_{k_n}}\right)\bigg/\frac{\log \om_{k'_n+1,k'_n+l'_n}}{\log \om_{k'_n}}&\ge\f{\log \om_{k'_n}}{\log \om_{k'_n,k'_n+l'_n}}\varphi\left(\alpha\right) 
			>\f{\log \om_{k'_n}}{\log \frac{\beta}2}\varphi\left(\alpha\right).
		\end{align*}
		Thus $\mathcal{M}(J,\{n_k\},\{\ccc_k\})$ is not quasi-normal.
		
		(c) Let $A=\inf_k \oc_k$ and $B=\sup_k \oc_k$. Then for every integer $l>0$, we obtain
	\begin{align*}
			\sup_{k\ge 1}\varphi\left(\frac{\log \um_{k+1,k+l}}{\log \om_{k}}\right)\bigg/\frac{\log \om_{k+1,k+l}}{\log \um_{k}}\ge \sup_{k\ge  1}\f {k\log A}{l\log B}\varphi\left(\alpha\right)=\infty,
		\end{align*}
		and $\mathcal{M}(J,\{n_k\},\{\ccc_k\})$ is not quasi-normal.
	\end{proof}

	\iffalse
	\begin{corollary}\label{cor_quasihomo}
		Let $\mathcal{M}(J,\{n_k\},\{\ccc_k\})$ satisfy $c_*>0$. Then it is normal. 
	\end{corollary}
	\begin{proof}
		Since  $c_*>0$, it is clear that $
		\sup_{k,l}\f{\log \um_{k,k+l}}{\log \om_{k,k+l}}\le \f{\log c_*}{\log (1-c_*)},
		$
		and  the conclusion immedaitely follows by Proposition \ref{quasihomo} (i) (a).
	\end{proof}
	\fi

	\section{General lower bounds of quasi-Assouad dimensions}\label{sec_glb}
	 In this section, we show that $t_*$, defined in \eqref{def_t*tt*}, provides a lower bound for the quasi-Assouad dimension of general Moran sets.  First, we recall three well-known facts, and refer the reader to \cite{CM, Wen2000}, \cite[Lemma 9.2]{falconer2013fractal}, and \cite[Proposition 2.2]{hua2000structures} for their proofs.
	\begin{lemma}\label{compensate}
	Let $\tilde{D}\subset D^*$ be such that $\{J_{\bu}:\bu\in\tilde{D}\}$ forms a finite, non-overlapping covering of the Moran set $E$. Then for any $s>0$, there exist $k_1,k_2$ with $\min_{\bu\in\tilde{D}}|\bu|\leq k_1,k_2\leq \max_{\bu\in\tilde{D}}|\bu|$ such that
		\[
		\sum_{\bu\in{D}_{k_1}}c^{s}_{\bu}\leq\sum_{\bu\in\tilde{D}}c^{s}_{\bu}\leq\sum_{\bu\in{D}_{k_2}}c^{s}_{\bu}.
		\]
	\end{lemma}
	
\begin{lemma}\label{numberofball}
	Let $\{V_i\}$ be a collection of disjoint open subsets of $\R^n$ such that each $V_i$ contains a
	ball of radius $a_1r$ and is contained in a ball of radius $a_2r$. Then any ball $B$ of radius
	$r$ intersects at most $(1+2a_2)^na_1^{-n}$ of the closures $\overline{V}_i$.
\end{lemma}

\begin{proposition}\label{lowerboundofbox}
	For any $E\in \mathcal{M}(J,\{n_k\},\{\ccc_k\})$, we have
	\[
	\odimb E\ge \limsup_{k\to\infty} s_{1,k}.
	\]
\end{proposition}

	Next, we prove that  $t_*$ is always a lower bound for the quasi-Assouad dimension.
	
	\begin{theorem}\label{dayu}
For any $E\in \mathcal{M}(J,\{n_k\},\{\ccc_k\})$,  $\dimqa E\ge  t_*.$
	\end{theorem}
	\begin{proof}%[Proof of Theorem \ref{dayu}]
		Without loss of generality, we assume that $\lim_{k\to\infty} \om_k = 0$, since otherwise $E$ has non-empty interior and $\dimqa E = d$, so the conclusion holds.
		
Fix $\eta>0$. For each $s < \limsup_{l \to \infty} \sup_{k \in \uk_{l,\eta}} s_{k+1,k+l}$, there exist two sequences $\{p_k\}$ and $\{q_k\}$ with $p_k \in \uk_{q_k,\eta}$ and $q_k \to \infty$, such that $s < s_{p_k+1,p_k+q_k}$ for all $k \ge 1$, and 
		\begin{gather*}
			\lim_{k\to\infty}s_{p_k+1,p_k+q_k}= \limsup_{l\to\infty}\sup_{k\in \uk_{l,\eta}} s_{k+1,k+l}.
		\end{gather*}
		%Since $\sup_{l\ge 1}\lim_{k\to\infty}\frac{\log \om_{k+l}-\log \om_{k}}{\log d_1d_2\cdots \uc_k}=0$, we have
If $\{p_k\}$ has a bounded subsequence $\{p_{k_t}\}$, then, since $\lim_{k \to \infty} \om_k = 0$, there exists $M>0$ such that
		\[
		\lim_{t\to\infty} s_{p_{k_t}+1,p_{k_t}+q_{k_t}}\le  \limsup_{l\to\infty}\max_{1\le k\le M} s_{k,k+l}=\limsup_{l\to\infty} s_{1,l}.
		\]
		By \eqref{dimHBAQ} and Proposition \ref{lowerboundofbox}, we obtain
		\[
		\dimqa E\ge \overline{\text{dim}}_{\rm B}E\ge \limsup_{l\to\infty} s_{1,l}> s,
		\]
	and the result follows.

	 Otherwise, $\lim_{k \to \infty} p_k = +\infty$, and by \eqref{dimHBAQ} and \eqref{def_t*tt*}, it suffices to prove that
$$
h_F(\eta) \geq s. 
$$
Let $k > 0$ be an integer. For each $m \in \N$, we define
		\[
		\mathcal{B}_{m,k}=\left\{\jjj\in D_{p_k+1,p_k+q_k}:2^{-m-1}<c_{\jjj}\le 2^{-m}\right\},
		\]
		and set
		\[
		m_k=\min\left\{m:\mathcal{B}_{m,k}\neq \emptyset\right\}.
		\]
		Since $p_k \in \uk_{q_k,\eta}$, this implies that
\begin{equation}\label{M/M}
				2^{-m_k-1}<\max_{\jjj\in D_{p_k+1,p_k+q_k}}c_{\jjj}
				=\om_{p_k+1,p_k+q_k} <\om_{p_k}^\eta,
		\end{equation}
		and  $	\lim_{k\to\infty}m_k=\infty. $
		It follows that 
		$$
		\sum_{m=0}^{\infty}\#\mathcal{B}_{m,k}2^{-ms}\ge  \sum_{m=0}^{\infty}\sum_{\jjj\in\mathcal{B}_{m,k}}c_{\jjj}^s =\sum_{\jjj\in D_{p_k+1,p_k+q_k}}c_{\jjj}^s =\prod_{i=p_k+1}^{p_k+q_k}\sum_{j=1}^{n_i}c_{i,j}^s
		$$
		Since   $s<s_{p_k+1,p_k+q_k}$, we have that 
		\begin{equation} \label{numberofB}
			\sum_{m=0}^{\infty}\#\mathcal{B}_{m,k}2^{-ms} >1.
		\end{equation}
		
		Given $\eps > 0$, for each $k > 0$, there exists an integer $m_k' \ge m_k$ such that
		\begin{equation}\label{numberofB1}
			2^{-\eps m_k'}(1-2^{-\eps})\le \#\mathcal{B}_{m_k',k}2^{-m_k's}.
		\end{equation}
	Otherwise  there exists an integer $k>0$ such that $\#\mathcal{B}_{m,k}2^{-ms}<2^{-\eps m}(1-2^{-\eps})$ for all $ m\in \N$, and it follows that 
		\[
		\sum_{m=0}^{\infty}\#\mathcal{B}_{m,k}2^{-ms}<\sum_{m=0}^{\infty}2^{-\eps m}(1-2^{-\eps})=1,
		\]
		which 	contradicts \eqref{numberofB}.  Moreover $	\lim_{k\to\infty}m_k'=\infty. $
		
	    Choose $\iii\in D_{p_k}$ such that $c_\iii=\om_{p_k}$. Set
		\[
		R_k=c_{\iii},\quad r_k=\min_{\jjj\in \mathcal{B}_{m_k',k}}c_{\iii*\jjj}.
		\]
		By \eqref{M/M},  it is clear that
		\begin{align*}
			\frac{r_k}{R_k}&=\min_{\jjj\in \mathcal{B}_{m_k',k}}c_{\jjj} \le {\om_{p_k+1,p_k+q_k}} < \om_{p_{k}}^\eta = R_k^{\eta},
		\end{align*}
		and we have   $r_k< R_k^{1+\eta}$. For every $\jjj \in \mathcal{B}_{m_k',k}$, $\mbox{int}(J _{\iii*\jjj})$ contains a ball of    radius $2^{-m_k'-2} c_{\iii}$ (which is at least $\frac14 r_k$) and is contained in a ball of radius $2^{-m_k'-1} c_{\iii}$ (which is less than $r_k$).    Furthermore, the elements of the set $\{\mbox{int}(J _{\iii*\jjj}):\jjj \in \mathcal{B}_{m_k',k}\}$ are pairwise disjoint.  
		It then follows from Lemma \ref{numberofball} that
		\[
		\sup_{x}\#\left\{\jjj\in \mathcal{B}_{m_k',k}:B(x,r_k)\cap J_{\iii*\jjj}\neq \emptyset\right\}\le 12^d.
		\]
		
		For each integer $k\geq 1$, we write
		\[
		t_k=\sup_{x\in J_{\iii}\cap E}N_{r_k}\left(B(x,R_k)\cap E\right).
		\]
		Then for any $z\in J_{\iii}\cap E$, there exist $x_1,\ldots,x_{t_k}\in J_\iii\cap E$\ (depending on $z$) such that
		\[
		\bigcup_{\jjj\in \mathcal{B}_{m_k',k}}J_{\iii*\jjj}\subset J_{\iii}\subset B(z,R_k)\subset \bigcup_{i=1}^{t_k}B(x_i,r_k).
		\]
		Hence for each $\jjj\in \mathcal{B}_{m_k',k}$, there exists  $B(x_i,r_k)$ such that $J_{\iii*\jjj}\cap B(x_i,r_k)\neq \emptyset$. Therefore,
		\[
		\mathcal{B}_{m_k',k}=\bigcup_{i=1}^{t_k}\left\{\jjj\in \mathcal{B}_{m_k',k}:B(x_i,r_k)\cap J_{\iii*\jjj}\neq \emptyset\right\}, 
		\]
		and we have 
		\begin{align}\label{numberofB2}
			\#\mathcal{B}_{m_k',k}\le12^d t_k\le 12^d \sup_{x\in E}N_{r_k}\left(B(x,R_k)\cap E\right).
		\end{align}
		
		For each  $\de>0$, by \eqref{def_hED}, there exists $C_{\de}$ such that for all $k>0$,
		\[
		\sup_{x\in E}N_{r_k}\left(B(x,R_k)\cap E\right)\le C_{\de}\Big(\frac{R_k}{r_k}\Big)^{h_E(\eta)+\de}. 
		\]
		Combining \eqref{numberofB1} and \eqref{numberofB2}, we have that for all $k>0$, 
		\begin{align*}
			12^{-d}2^{(s-\eps) m_k'}(1-2^{-\eps})&\le 12^{-d}\#\mathcal{B}_{m_k',k} 
			\le C_{\de}\Big(\frac{R_k}{r_k}\Big)^{h_{E}(\eta)+\de} \le C_{\de}2^{(m_k'+1)(h_{E}(\eta)+\de)}.
		\end{align*}
Hence,  $	h_{E}(\eta)+\de\ge s-\eps, $ and we have $h_E(\eta)\ge s$ by the arbitrariness of  $\de, \eps$. 
	\end{proof}
	
	\section{quasi-Assouad dimension of Moran sets }\label{sec_qadms}
	
In this section, we derive bounds for the quasi-Assouad dimension of a Moran set subject to either \eqref{club} or the BBC. 
	 The first lemma shows that \eqref{club} implies a weaker form of the BNC.  
	\begin{lemma}\label{numberofcap}
Suppose  that $\mathcal{M}(J,\{n_k\},\{\ccc_k\})$ satisfies \eqref{club}.  Then, for all $E\in\mathcal{M}(J,\{n_k\},\{\ccc_k\})$,  the limit
		\[
		\lim_{\de\to 0}\frac{\log \sup_{x\in E}\#\{\bu\in D_{\iii}(\de):B(x,c_{\iii}\de)\cap J_{\bu}\neq \emptyset\}}{-\log c_{\iii}\de}=0
		\]
		holds uniformly in  $\iii\in D^*$. In particular, if $\iii$ is the empty word, we have 
	\begin{equation}\label{AWSC}
		\lim_{\de\to 0}\frac{\log \sup_{x\in E}\#\{\bu\in D(\de):B(x,\de)\cap J_{\bu}\neq \emptyset\}}{-\log\de}=0.
		\end{equation}
	\end{lemma}
	
	\begin{proof}
		Given $\iii\in D^*$ and  $\de>0$, there exists  $\jjj\in D_\iii(\de)$ such that 
		\[
		\uc_{|\jjj|}=\min\{\uc_{|\bu|}:\bu\in D_{\iii}(\de)\}.
		\]
		Observe that  $\lim_{\delta \to 0} |\jjj| =\infty$. For every $\bu\in D_{\iii}(\de)$, it is clear that 
		\[
		\frac{c_{\bu}}{c_{\iii}}\le \de<\frac{c_{\bu^-}}{c_\iii}\le\frac{c_{\bu}}{c_\iii \uc_{|\bu|}}\le\frac{c_{\bu}}{c_{\iii}\uc_{|\jjj|}}.
		\]
		For every $\bu\in D_{\iii}(\de)$, $\mbox{int}(J_{\bu})$ contains a ball of  radius $\frac{c_\iii \uc_{|\jjj|}}2\de$ and is contained in a ball of radius $\frac{c_\iii}2\de$.      Furthermore, the elements of the set $\{\mbox{int}(J_\bu):\bu\in D_\iii(\de)\}$ are pairwise disjoint.  
		By Lemma \ref{numberofball}, we have
		\[
		\sup_{x\in E}\#\{\bu\in D_{\iii}(\de):B(x,c_\iii\de)\cap J_{\bu}\neq \emptyset\}\le 4^d  \uc_{|\jjj|}^{-d}.
		\]
		Since  $	c_\iii\de<\frac{c_{\jjj}}{\uc_{|\jjj|}}\le\frac{\om_{|\jjj|}}{ \uc_{|\jjj|}},  $   we have 
		\begin{equation*}
			\lim_{\de\to 0}\frac{\log \sup_{x\in E}\#\{\bu\in D_{\iii}(\de):B(x,c_{\iii}\de)\cap J_{\bu}\neq \emptyset\}}{-\log c_\iii\de}\le\lim_{|\jjj|\to\infty}\frac{d\,(\log4-\log \uc_\jjj)}{\log \uc_\jjj-\log \om_{|\jjj|}}=0
		\end{equation*}
		uniformly for $\iii\in D^*$.
	\end{proof}

	\begin{remark}
	Condition \eqref{AWSC}, first introduced in \cite{AWSC1} as the asymptotically weak separation condition, was subsequently applied in \cite{AWSC2,AWSC3}.
	\end{remark}

	\begin{theorem}\label{xiaoyu}
		Let $ \mathcal{M}(J,\{n_k\},\{\ccc_k\})$ satisfy \eqref{club}. Then every  $E\in\mathcal{M}(J,\{n_k\},\{\ccc_k\})$ satisfies
		\[
		\dimqa E\le t.
		\]
	\end{theorem}
	\begin{proof}%[Proof of Theorem \ref{xiaoyu}]
		Fix $\eta>0$, and  set  $s^{*}(\eta)=\limsup_{l\to\infty}\sup_{k\in \kk_{l,\eta}} s_{k+1,k+l}$. By \eqref{def_t*tt*}, it is clear that $\lim_{\eta \to 0} s^{*}(\eta)=t^*$,  and it suffices to prove that
		\[
		h_E(2\eta)\le s^*(\eta).
		\]
		
		Arbitrarily choose $\eps>0$.	Since \eqref{club} holds, it folows from
		\[
		\f{\f{\log \uc_{k+1}}{\log \om_k}}{1+\f{\log \uc_{k+1}}{\log \om_k}}\le\f{\log \uc_{k+1}}{\log \om_{k+1}}
		\]
		 that $\lim_{k\to\infty}\f{\log \uc_{k+1}}{\log \om_k}=0$. Hence, by induction, we also have $\lim_{k\to\infty}\f{\log \uc_{k+l}}{\log \om_k}=0$ for any $l$.   Note that  $\sum
		_{j=1}^{n_{k}}c_{k,j}^{\,d}\leq 1$ for all $k>0$, it follows that
		\[
		\f{\log n_{k+1}\cdots n_{k+l}}{-\log \om_{k-1}}\le\f{d\log \uc_{k+1}\cdots \uc_{k+l}}{\log \om_{k-1}} .
		\]
Combining these with  Lemma \ref{numberofcap},
		there exist $k_0, k'_0 \in {\N}^+,$ and $\de_0\in (0,\eta)$ such that
		\begin{eqnarray}
			\sup_{k\in \kk_{l,\eta}} s_{k+1,k+l}<s^*(\eta)+2\eps, \qquad &&l\ge k_0,\label{sk}   \\
			\f{\log \uc_k}{\log \om_k}<\frac{1}{2},\qquad  \frac{\om_{|\iii|+k}^{\eps}}{\uc_{|\iii|}^{s^*(\eta)+2\eps}\uc_{|\iii|+k}^{s^*(\eta)+2\eps}}<1,\qquad && k\ge k_0,\ \iii\in D^*,\vspace{0.5em}\label{dkmk}\\
			n_{k+1}n_{k+2}\cdots n_{k+k_0-1}<\om_{k-1}^{- {\eps}},\qquad &&k\ge k'_0,\label{nk}\vspace{0.2em}\\
			\sup_{x\in E}\#\{\bu\in D(\delta):B(x,\de)\cap J_{\bu}\neq \emptyset\}<\de^{-{\eps}},\qquad &&0<\de<\de_0.\label{numberofcut'}
		\end{eqnarray}
		Given $\iii\in D^*$, for each $k\in{\N}^+$ and $\delta>0$, we write
		\[
		D_{\iii}(\de,k)=\{\iii*\jjj\in D^*:c_{\jjj}\le \de<c_{\jjj^-},|\jjj|=k\}
		\]
		and $D(\de,k)=D_{\emptyset}(\de,k)$. For each $\bu=\iii*\jjj\in D_{\iii}(\de,k)$, it is clear that
		$$
		\de^{s^*(\eta)+2\eps}< c_{\jjj^-}^{s^*(\eta)+2\eps} \leq \Big(\frac{c_{\bu}}{c_{\iii}\uc_{|\iii|+k}}\Big)^{s^*(\eta)+2\eps} \le  \frac{c_{\bu}^{s^*(\eta)+\eps}\om_{|\iii|+k}^{\eps}}{c_{\iii^-}^{s^*(\eta)+2\eps}\uc_{|\iii|}^{s^*(\eta)+2\eps}\uc_{|\iii|+k}^{s^*(\eta)+2\eps}}.
		$$
		For $k\ge k_0$,  by \eqref{dkmk}, it follows that
		\begin{equation}\label{estdelta1}
			\begin{aligned}
				\de^{s^*(\eta)+2\eps}&< {c_{\iii^-}^{-(s^*(\eta)+2\eps)}}  c_{\bu}^{s^*(\eta)+\eps} .
			\end{aligned}
		\end{equation}

		Given $r,R$ satisfying $0<r<R^{1+2\eta}<R<\min \{ \de_0,\om_{k_0},\om_{k'_0}\}$,  for every $\iii\in D(R)$  and $\bu\in D_\iii\left(\f rR\right)$, we have $|\iii|\ge \max\{k_0,k'_0\}$ and $\frac{c_\bu}{c_\iii}\leq \frac{r}{R}$. By \eqref{dkmk}, we have
		$$
		\log \om_{|\iii|}\geq 2(\log \om_{|\iii|}-\log\uc_{|\iii|})\geq 2(\log c_{\iii}-\log\uc_{|\iii|}) \geq 2\log c_{\iii^-} \geq 2\log R.
		$$
		Immediately, it follows that 
		\begin{align*}
			\f{\log \um_{|\iii|+1,|\bu|}}{\log \om_{|\iii|}}
			&\ge \f{\log{c_\bu}-\log{c_\iii}}{\log \om_{|\iii|}} >\f{\log\f{r}R}{2\log R} >\eta, 
		\end{align*}
		and this implies  $|\iii|\in \kk_{|\bu|-|\iii|,\eta}$.

		For every $\iii\in D(R)$, define $k_1=\max\big\{k:D_\iii(\f rR,k)\neq\emptyset\big\} $. Then we have the following two cases.
		
		\noindent	Case A: If $k_1 <k_0$,  by \eqref{nk}, we have
		\begin{align*}
			\# D_{\iii}\Big(\f rR\Big)
			&=\sum_{k=1}^{k_1}\# D_{\iii}\Big(\f rR,k\Big)\ \le n_{|\iii|+1}n_{|\iii|+2}\cdots n_{|\iii|+k_0-1} <\om_{|\iii|-1}^{-{\eps}} <c_{\iii^-}^{-{\eps}}.
		\end{align*}
		Case B: If $k_1 \ge k_0$, set $	k'_1=\min\big\{k: D_\iii(\f rR,k)\neq\emptyset,k\ge k_0\big\}$, and  we obtain
		\begin{equation}\label{numberofcut}
			\# D_{\iii}\Big(\f rR\Big)
			=\sum_{k=1}^{k_0-1}\# D_{\iii}\Big(\f rR,k\Big)+\sum_{k=k'_1}^{k_1}\# D_{\iii}\Big(\f rR,k\Big)<c_{\iii^-}^{-{\eps}}+\sum_{k=k'_1}^{k_1}\#D_{\iii}\Big(\f rR,k\Big).
		\end{equation}
	Let
		\begin{align*}
			Q=\Big(\left\{J_\bu:\bu\in D_\iii\Big(\f rR,k'_1\Big)\right\}\Big\backslash \Big\{J_{\bu}:\bu=\mathbf{\tau}|_{|\iii|+k'_1},\mathbf{\tau}\in\bigcup_{k=k'_1}^{k=k_1}D_{\iii}\Big(\f rR,k\Big)\Big\}\Big)\\
			\bigcup \Big\{J_{\bu}:\bu\in \bigcup_{k=k'_1}^{k=k_1}D_{\iii}\Big(\f rR,k\Big)\Big\}.
		\end{align*}
	Then $Q$ is a finite, non-overlapping covering of $J_\iii \cap E$, and by \eqref{estdelta1}, we have 
		$$	
		\sum_{k=k'_1}^{k_1}\# D_{\iii}\Big(\f rR,k\Big)\Big(\f rR\Big)^{s^*(\eta)+2\eps} <\frac1 {c_{\iii^-}^{s^*(\eta)+2\eps}}\sum_{k=k'_1}^{k_1}\sum_{\bu\in D_{\iii}(\f rR,k)}c_{\bu}^{s^*(\eta)+2\eps} 
		\le\frac1 {c_{\iii^-}^{s^*(\eta)+2\eps}}\sum_{J_\bu\in Q}c_{\bu}^{s^*(\eta)+2\eps}
		$$	
		By Lemma \ref{compensate},  there exists  $k_1'\leq k^*\leq k_1$ such that  
		$$
		\sum_{J_\bu\in Q}c_{\bu}^{s^*(\eta)+2\eps}\leq c_{\iii}^{s^*(\eta)+2\eps}\sum_{\jjj\in  D_{|\iii|+1,|\iii|+k^*}}c_{\jjj}^{s^*(\eta)+2\eps}= c_{\iii}^{s^*(\eta)+2\eps} \prod_{k=|\iii|+1}^{|\iii|+k^*}\sum_{j=1}^{n_k}c_{k,j}^{s^*(\eta)+2\eps}.
		$$
		Note that $|\iii|\in \kk_{k'_1,\eta}\subseteq \kk_{k^*,\eta}$.   Combining these with	\eqref{sk}, we have that
		\begin{equation*} %\label{estdelta2}
			\begin{aligned}
				\sum_{k=k'_1}^{k_1}\# D_{\iii}\Big(\f rR,k\Big)\Big(\f rR\Big)^{s^*(\eta)+2\eps}
				&<c_{\iii^-}^{-\eps}\prod_{k=|\iii|+1}^{|\iii|+k^*}\sum_{j=1}^{n_k}c_{k,j}^{s^*(\eta)+2\eps}<c_{\iii^-}^{-\eps}.
			\end{aligned}
		\end{equation*}  
		By  \eqref{numberofcut}, it follows that
		\begin{equation*}
			\#D_{\iii}\Big(\f rR\Big)< c_{\iii^-}^{-{\eps}}+c_{\iii^-}^{-{\eps}}\Big(\f Rr \Big)^{s^*(\eta)+2\eps}< 2c_{\iii^-}^{-{\eps}}\Big(\f Rr\Big)^{s^*(\eta)+2\eps}.
		\end{equation*}
		
		Combining Case A and Case B, we obtain that
		\begin{equation}\label{numberofcut2}
			\#D_{\iii}\left(\f rR\right)\le 2c_{\iii^-}^{-{\eps}}\Big(\f Rr\Big)^{s^*(\eta)+2\eps} \le 2R^{-{\eps}}\Big(\f Rr\Big)^{s^*(\eta)+2\eps},
		\end{equation}
		for $0<r<R^{1+2\eta}<R<\min \{ \de_0,\om_{k_0},\om_{k'_0} \}$ and $\iii\in D(R)$.
		\iffalse
		Obviously, $\# D(\de)\de^{s^*(\eta)+\eps}<\# D(\de)\le \# D_{k_0}\le N^{k_0}$ when $\om_{k_0}\le\de<1$. It follows that
		\begin{equation}\label{numberofcutset1}
			\# D_{\iii}(\de)=\# D(\de)<N^{k_0}\de^{-s^*(\eta)-\eps}
		\end{equation}
		for all $\iii\in D$ and all $0<\de<1$.
		\fi 
		
		Fix $x\in E$. Next we estimate $N_r(B(x,R)\cap E)$ for $0<r<R^{1+2\eta}<R<\min\left\{ \de_0,\om_{k_0},\om_{k'_0}\right\}$. It is clear that
		\[
		B(x,R)\cap E\subset \bigcup_{\iii\in D(R),J_{\iii}\cap B(x,R)\cap E\neq \emptyset} J_{\iii}\cap E.
		\]
		For each $\iii\in D(R)$ with $J_{\iii}\cap B(x,R)\cap E\neq \emptyset$, we have
		\[
		J_{\iii}\cap E\subseteq\bigcup_{\bu\in D_{\iii}\left(\frac r R\right)} J_{\bu}.
		\]
		For each $\bu\in D_{\iii}\left(\frac r R\right)$, choose a point $x_\bu\in J_\bu\cap E$. Then $	J_{\bu}\subseteq B(x_{\bu},r) $ since  $c_{\bu}\le  r$ and $|J|=1$.
		Hence, it follows that 
		\[
		B(x,R)\cap E\subset\bigcup_{\iii\in D(R),J_{\iii}\cap B(x,R)\cap E\neq \emptyset}\ \bigcup_{\bu\in D_{\iii}\left(\frac r R\right)}B(x_{\bu},r).
		\]
		Combining it with \eqref{numberofcut2}, we have 
		\begin{align*}
			N_r(B(x,R)\cap E)&\le \sum_{\substack{\iii\in D(R)\\J_{\iii}\cap B(x,R)\cap E\neq \emptyset}}\# D_{\iii}\left(\frac r R\right)<2R^{-{\eps}}\sum_{\substack{\iii\in D(R)\\J_{\iii}\cap B(x,R)\cap E\neq \emptyset}}\Big(\frac R r\Big)^{s^*(\eta)+2\eps} 
		\end{align*}	
		Since $R^{-\eta}<\f Rr$, by \eqref{numberofcut'}, it follows that 
		\begin{align*}
			N_r(B(x,R)\cap E)
			&<2R^{-{\eps}}\Big(\frac R r\Big)^{s^*(\eta)+2\eps}\sup_{x\in E}\#\{\iii\in D(R):B(x,R)\cap J_{\iii}\cap E\neq \emptyset\}\\
			&\le 2R^{-{\eps}}\Big(\frac R r\Big)^{s^*(\eta)+2\eps}\sup_{x\in E}\#\{\iii\in D(R):B(x,R)\cap J_{\iii}\neq \emptyset\}\\
			&<2\Big(\frac R r\Big)^{s^*(\eta)+2\eps+\f{\eps}{\eta}}.
		\end{align*}
		Therefore, $h_E(2\eta)\le s^*(\eta)+2\eps+\f{\eps}{\eta}.$  Letting $\eps\to 0$ completes the proof.
	\end{proof}

	\begin{proposition}\label{BC->BNC}
		Let $\mathcal{M}(J,\{n_k\},\{\ccc_k\})$ satisfy the BBC. Then every $E\in\mathcal{M}(J,\{n_k\},\{\ccc_k\})$ satisfies the BNC.
	\end{proposition}
	\begin{proof}
		Given  $\de>0$, let	$\mathcal{S}(\de)=\big\{\bu^-:\bu\in D(\de)\big\}$. For $\bv\in \mathcal{S}(\de)$,  remove $\bv$ from $\mathcal{S}(\de)$ if  $\bv*\iii\in \mathcal{S}(\de)$ for some $\iii$, and we denote the resulting set by $\mathcal{S}'(\de)$. Note that $|J_\bu|=c_\bu>\de$ for any $\bu\in \mathcal{S}'(\de)$ and the elements of $\{\mbox{int}(J_\bu):\bu\in \mathcal{S}'(\de)\}$ are pairwise disjoint.

		Since in $\mathbb{R}^d$, a ball of radius $\de$ can intersect at most finitely many pairwise disjoint balls whose radii exceed $\frac\de2$, there exists a constant $N_d$ depending only on the ambient dimension $d$ such that
		\begin{equation*}
			\sup_{\de>0,x\in E} \#\big\{\bu\in \mathcal{S}'(\de): E\cap J_\bu\cap B(x,\de)\neq \emptyset \big\}<N_d.
		\end{equation*}
		It then follows that
		\begin{align*}
			\sup_{x\in E} \#\big\{\bu\in D(\de):E\cap J_\bu\,\cap &\,B(x,\de)\neq \emptyset \big\}\\
			&< N_d\sup_{\bu\in D(\de)}\#\big\{\bu'\in D(\de):\bu'=\bu^-*\iii\text {\ for some }\iii\big\}.
		\end{align*}
		Since  $\mathcal{M}(J,\{n_k\},\{\ccc_k\})$ satisfies the BBC,  the conclusion holds by taking $\de\to 0$.
	\end{proof}

	\begin{theorem}\label{xiaoyu2}
		Let $ \mathcal{M}(J,\{n_k\},\{\ccc_k\})$ satisfy BBC. Then every $E\in\mathcal{M}(J,\{n_k\},\{\ccc_k\})$ satisfies
		\[
		\dimqa E\le t^*.
		\]
	\end{theorem}
	\begin{proof}
		For any $\iii\in D^*$ and $\de\in (0,1)$, define $\mathcal{S}_\iii(\de)=\{\bu^-:\bu\in D_{\iii}(\de)\}$. For $\bv\in \mathcal{S}_\iii(\de)$, remove $\bv$ from $\mathcal{S}_\iii(\de)$ if  $\bv*\jjj\in \mathcal{S}_\iii(\de)$ for some $\jjj$, and denote the resulting set by $\mathcal{S}_\iii'(\de)$. Since BBC holds, each element of $\mathcal{S}_\iii'(\de)$ contains at most $N$ elements of $D_\iii(\de)$. Moreover, $\{J_\bu:\bu\in \mathcal{S}_\iii'(\de)\}$ is a finite nonoverlapping covering of Moran set $J_\iii\cap E$. Therefore, by Lemma \ref{compensate}, for any $s>0$, it follows that
		\begin{equation}\label{numberofD_ide}
			\# D_\iii(\de)\de^s<c_\iii^{-s}\sum_{\bu\in D_\iii(\de)}c_{\bu^-}^s\le Nc_\iii^{-s}\sum_{\bu\in \mathcal{S}_\iii'(\de)}c_{\bu}^s\le N \prod_{k=|\iii|+1}^{|\iii|+k^*}\sum_{j=1}^{n_k}c_{k,j}^s,
		\end{equation}
		where $k^*\ge \min\{|\bu|-|\iii|:\bu\in \mathcal{S}_\iii'(\de)\}=\min\{|\bu|-|\iii|:\bu \in D_\iii(\de)\}-1$.
			
Fix $\eta\in (0,1)$. For  $s>\limsup_{l\to\infty}\sup_{k\in \ok_{l+1,\eta}} s_{k,k+l}$, there exists $k_0>2$ such that
		\begin{equation}\label{K'}
			\sup_{k\in \ok_{l+1,\eta}} s_{k+1,k+l}<s,
		\end{equation}
 for $l\ge k_0$. By \eqref{dimHBAQ} and \eqref{def_t*tt*}, it suffices to prove that 
$$
h_F(\eta) \leq s. 
$$
	
Since $ \mathcal{M}(J,\{n_k\},\{\ccc_k\})$ satisfies  BBC,  by Propositions \ref{BC} and \ref{BC->BNC}, there exist reals $M>0$ and $\de_0>0$ such that $\sup_{k\ge 1} n_k<M$ and 
	\begin{gather}
		\sup_{x\in E} \#\big\{\bu\in D(\de):E\cap J_\bu\,\cap \,B(x,\de)\neq \emptyset \big\}<M,\qquad 0<\de<\de_0.\label{bnc}
	\end{gather}
		Given $r,R$ satisfying $0<r<R^{1+\eta}<R<\de_0$,  for every $\iii\in D(R)$  and $\bu\in D_\iii\left(\f rR\right)$, we have $c_{\iii^-}>R$ and $\frac{c_\bu}{c_\iii}\le \frac{r}{R}<R^\eta$. Hence,
		\[
		\f{\log \um_{|\iii|+1,|\bu|}}{\log \om_{|\iii|-1}}\ge \f{\log c_\bu-\log c_\iii}{\log c_{\iii^-}}>\frac{\log r-\log R}{\log R}>\eta,
		\]
		which implies that $|\iii|\in\ok_{|\bu|-|\iii|,\eta}$. Therefore, 
		\[
		|\iii|\in\ok_{\min\left\{|\bu|-|\iii|:\,\bu\in D_\iii(\f r R)\right\},\eta}\subseteq\ok_{k^*+1,\eta},
		\]
		where $k^*$ is given by \eqref{numberofD_ide}.
		
	In addition, for every $\iii\in D(R)$, we distinguish two cases:
		
		\noindent	Case A: If $k^* < k_0$,  by  \eqref{numberofD_ide} and $\sup_{k\ge 1} n_k<M$, we have
		\begin{align*}
			\# D_{\iii}\Big(\f rR\Big)\Big(\f rR\Big)^s
			& \le NM^{k^*}\om_{|\iii|+1,|\iii|+k^*}^s< NM^{k_0}.
		\end{align*}
		Case B: If $k^* \ge k_0$, recall that $|\iii|\in\ok_{k^*+1,\eta}$. Then by \eqref{numberofD_ide} and \eqref{K'}, we have
		$$	
		\# D_{\iii}\Big(\f rR\Big)\Big(\f rR\Big)^s<N.
		$$	
	
	Combining Case A and Case B, we obtain that
		\begin{equation*}%\label{numberofD_ide'}
			\#D_{\iii}\Big(\f rR\Big)<NM^{k_0}\Big(\f Rr\Big)^s.
		\end{equation*}
By an argument analogous to that used in the proof of Theorem \ref{xiaoyu}, we obtain that,	for all $x\in E$, 
$$
N_r(B(x,R)\cap E)\le \sum_{\substack{\iii\in D(R)\\J_{\iii}\cap B(x,R)\cap E\neq \emptyset}}\# D_{\iii}\left(\frac r R\right)<NM^{k_0}\sum_{\substack{\iii\in D(R)\\J_{\iii}\cap B(x,R)\cap E\neq \emptyset}}\Big(\frac R r\Big)^s
$$
		Finally, combining it with \eqref{bnc}, we obtain that 
		\begin{align*}
			\sup_{x\in E}N_r(B(x,R)\cap E) &<NM^{k_0}\sup_{x\in E}\#\{\bu\in D(R):E\cap J_{\bu}\cap B(x,R)\neq \emptyset\}\Big(\frac R r\Big)^s\\
			&<NM^{k_0+1}\Big(\frac R r\Big)^s. 
		\end{align*}
	 Therefore, $h(\eta)\le s$, and the desired inequality $h_F(\eta) \leq s$ follows.
	\end{proof}
	
	\begin{lemma}\label{t=t^*}
For the Moran structure  $\mathcal{M}(J,\{n_k\},\{\ccc_k\})$, the following statements hold.
	\begin{enumerate}[(i)]
		\item If $\lim_{l\to \infty}\sup_{k\ge 1}\f{\log \uc_k}{\log \om_{k+1,k+l}}=0$, then for all $\eta>0$,
		\[
		\limsup_{l\to \infty}\sup_{k\in \ok_{l,\eta}}s_{k+1,k+l}=\limsup_{l\to \infty}\sup_{k\in \ok_{l,\eta}}s_{k,k+l}.
		\]
		\item If $\lim_{l\to \infty}\sup_{k\ge 1}\f{\log \uc_{k+l+1}}{\log \om_{k+1,k+l}}=0$, then for all $\eta>0$,
		\[
		\limsup_{l\to \infty}\sup_{k\in \ok_{l+1,\eta}}s_{k+1,k+l}=	\limsup_{l\to \infty}\sup_{k\in \ok_{l+1,\eta}}s_{k+1,k+l+1}=\limsup_{l\to \infty}\sup_{k\in \ok_{l,\eta}}s_{k+1,k+l}.
		\]
	\end{enumerate}
	\end{lemma}
	\begin{proof}
We prove only (i), as the argument for (ii) follows in a similar manner.
		
	Recall that $s_k$ denotes $s_{k,k}$ for any $k$. It is straightforward that, for any $k,l$,
	\begin{equation}\label{s_k,l}
	\min\{s_{k},s_{k+1,k+l}\}\le	s_{k,k+l}\le \max\{s_{k},s_{k+1,k+l}\}.
	\end{equation}
	Hence, we distinguish the following cases.
	
	 Case A. If $s_{k}=s_{k+1,k+l}$, then \eqref{s_k,l} holds with equality.
	 
	 Case B. If $s_{k}>s_{k+1,k+l}$, then, for any $0<\eps<s_{k}-s_{k+1,k+l}$, we have
	 \begin{align*}
	 \prod_{i=k}^{k+l}\sum_{j=1}^{n_i}c_{i,j}^{s_{k+1,k+l}+\eps}&=\Big(\sum_{j=1}^{n_k}c_{k,j}^{s_{k}}c_{k,j}^{s_{k+1,k+l}-s_{k}+\eps}\Big)\Big(\prod_{i=k+1}^{k+l}\sum_{j=1}^{n_i}c_{i,j}^{s_{k+1,k+l}+\eps}\Big)\\
	 &\le\uc_k^{s_{k+1,k+l}-s_{k}+\eps}\,\om_{k+1,k+l}^\eps\\
	 &= \uc_k^{s_{k+1,k+l}-s_{k}+\eps+\f{\log \om_{k+1,k+l}}{\log \uc_k}\eps}\to 0\ (l\to\infty),
	 \end{align*}
	where the last limit is uniform in $k$. By \eqref{s_k,l}, it follows that, for sufficiently large $l$, $s_{k+1,k+l}\le s_{k,k+l}<s_{k+1,k+l}+\eps$.
	
	Case C. If $s_{k}<s_{k+1,k+l}$, then for any $0<\eps<s_{k+1,k+l}-s_{k}$, we have
	\begin{align*}
		\prod_{i=k}^{k+l}\sum_{j=1}^{n_i}c_{i,j}^{s_{k+1,k+l}-\eps}&=\Big(\sum_{j=1}^{n_k}c_{k,j}^{s_{k}}c_{k,j}^{s_{k+1,k+l}-s_{k}-\eps}\Big)\Big(\prod_{i=k+1}^{k+l}\sum_{j=1}^{n_i}c_{i,j}^{s_{k+1,k+l}-\eps}\Big)\\
		&\ge\uc_k^{s_{k+1,k+l}-s_{k}- \eps}\,\om_{k+1,k+l}^{-\eps}\\
		&=\uc_k^{s_{k+1,k+l}-s_{k}+\eps-\f{\log \om_{k+1,k+l}}{\log \uc_k}\eps}\to +\infty\ (l\to\infty),
	\end{align*}
	where the last limit is uniform in $k$. By \eqref{s_k,l}, it follows that, for sufficiently large $l$, $s_{k+1,k+l}-\eps< s_{k,k+l}\le s_{k+1,k+l}$.
	
	Combining  the above cases, we in fact obtain
	\[
	\lim_{l\to \infty}\sup_{k\ge 1}\,(s_{k+1,k+l}-s_{k,k+l})=0.
	\]
Consequently, for any $\eta>0$,
	\begin{align*}
		\Big|\limsup_{l\to \infty}\sup_{k\in \ok_{l,\eta}}s_{k+1,k+l}-\limsup_{l\to \infty}\sup_{k\in \ok_{l,\eta}}s_{k,k+l}\Big|&\le \limsup_{l\to \infty}\sup_{k\in \ok_{l,\eta}}|s_{k+1,k+l}-s_{k,k+l}|\\
		&\le \limsup_{l\to \infty}\sup_{k\ge 1}|s_{k+1,k+l}-s_{k,k+l}|=0.\qedhere
	\end{align*}
	\end{proof}

	\begin{proof}[Proof of Theorem \ref{main1}]
		Since $\mathcal{M}(J,{n_k},{\ccc_k})$ satisfies either \eqref{club}, it follows from Theorems \ref{dayu} and \ref{xiaoyu} that it suffices to prove $t\le t_*$ when the structure is quasi-normal.
		
		Let  $\varphi:(0,+\infty)\to (0,+\infty)$ be a strictly increasing function such that 
		\[
		\limsup_{l\to \infty}\sup_k \varphi\Big(\frac{\log \um_{k+1,k+l}}{\log \om_{k}}\Big)\Big/\frac{\log \om_{k+1,k+l}}{\log \um_{k}}<C
		\]
		for  some $C>0$. Then  there exists $L>0$ such that  for $l>L$ and all $k>0$, 
		\[
		\varphi\Big(\frac{\log \um_{k+1,k+l}}{\log \om_{k}}\Big)<C\frac{\log \om_{k+1,k+l}}{\log \om_{k}}.
		\]
		Hence  
		$\kk_{l,\eta}\subseteq \uk_{l,\frac{\varphi(\eta)}C}$ for all $\eta>0$ and  all $l>L$. By \eqref{def_t*tt*}, we have that 
		\begin{equation*}
			\begin{aligned}
				t \le \lim_{\eta\to 0}\limsup_{l\to\infty}\sup_{k\in \uk_{l,\frac{\varphi(\eta)}C}} s_{k+1,k+l} \le\lim_{\eta\to 0}\limsup_{l\to\infty}\sup_{k\in \uk_{l,\eta}} s_{k+1,k+l} =t_*.
			\end{aligned}\qedhere
		\end{equation*}
		
	\end{proof}

\begin{proof}[Proof of Corollary \ref{t_*=t^*}]
	It is immediate that \eqref{cdnllZ} implies \eqref{club}. 
	Therefore, by Theorem \ref{main1}, it suffices to prove $t^*\le t$.
	
 By \eqref{cdnllZ} and Lemma \ref{t=t^*},   we obtain
\begin{align*}
	\limsup_{l\to \infty}\sup_{k\in \ok_{l+1,\eta}}s_{k+1,k+l}&=\limsup_{l\to \infty}\sup_{k\in \ok_{l,\eta}}s_{k+1,k+l}\\
	&=\limsup_{l\to \infty}\sup_{k\in \ok_{l,\eta}}s_{k,k+l}\\
	&=\limsup_{l\to \infty}\sup_{k+1\in \ok_{l,\eta}}s_{k+1,k+l+1}.
\end{align*}
Given $\eta>0$, fix $l>0$. For $k+1\in \ok_{l,\eta}$, since
\[
\f{\log \um_{k+1,k+l+1}}{\log \om_k}>	\f{\log \um_{k+2,k+l+1}}{\log \om_k}>\eta,
\]
we have  $k\in \kk_{l+1,\eta}$. It follows that
\[
\limsup_{l\to \infty}\sup_{k\in \ok_{l+1,\eta}}s_{k+1,k+l}	\le \limsup_{l\to \infty}\sup_{k\in \kk_{l+1,\eta}}s_{k+1,k+l+1}
=\limsup_{l\to \infty}\sup_{k\in \kk_{l,\eta}}s_{k+1,k+l}.
\]
Letting $\eta\to 0$ then completes the proof.
\end{proof}
	
	\begin{proof}[Proof of Theorem \ref{main2}]
		The proof is similar to that of Theorem \ref{main1}, and we omit it.
	\end{proof}

	\iffalse
	\begin{proof}[Proof of Theorem \ref{main2}]
		Since  $E$ is quasi-normal and satisfies \eqref{club},\,\ref{itm:index}, by Theorem \ref{xiaoyu} and Theorem \ref{dayu2}, it suffices to prove that $ t^*\le t$. 
		
		Let  $\varphi:(0,+\infty)\to (0,+\infty)$ be a strictly increasing function such that  
		\[
		\limsup_{l\to \infty}\sup_k \varphi\left(\frac{\log \um_{k+1,k+l}}{\log \om_{k}}\right)\bigg/\frac{\log \om_{k+1,k+l}}{\log \om_{k}}<C
		\]
		holds for some $C>0$. That is, there exists some $L$ such that 
		\[
		\varphi\left(\frac{\log \um_{k+1,k+l}}{\log \om_{k}}\right)<C\frac{\log \om_{k+1,k+l}}{\log \om_{k}}
		\]
		holds for $l>L$ and all $k$. Hence $ \ok_{l,\eta}\subseteq\kk_{l,\frac{\varphi(\eta)}C}$ for any $\eta>0$ and  $l>L$,  and by \eqref{def_t*tt*},  we have that 
		\begin{align*}
			t^*\le \lim_{\eta\to 0}\limsup_{l\to\infty}\sup_{k\in \kk_{l,\frac{\varphi(\eta)}C}} s_{k+1,k+l+1} \le\lim_{\eta\to 0}\limsup_{l\to\infty}\sup_{k\in \kk_{l,\eta}} s_{k+1,k+l+1} =t,
		\end{align*}
		completing the proof.
	\end{proof}
	\fi
	
	\begin{proof}[Proof of Theorem \ref{no.1}]
		This conclusion is an immediate consequence of either Corollary \ref{t_*=t^*} or Theorem \ref{main2}.
	\end{proof}

	\section{Assouad dimensions and Examples}\label{Sec_ADE}
	
	In this section, we study the Assouad dimensions of Moran sets.
The following corollary can be obtained directly from Proposition \ref{BC->BNC} and \cite[Theorem A]{kaenmaki2024regularity}.

	\begin{corollary}\label{cor_AD}
If $\mathcal{M}(J,\{n_k\},\{\ccc_k\})$  satisfies BBC, then every $E\in\mathcal{M}(J,\{n_k\},\{\ccc_k\})$ satisfies
		\[
		\dima E =\lim_{l\to\infty}\sup_{k\ge 1} s_{k+1,k+l}=\lim_{l\to\infty}\limsup_{k\to\infty}s_{k+1,k+l}=\inf_{l\ge 1}\limsup_{k\to\infty}s_{k+1,k+l}.
		\]
	\end{corollary}

Next, we give the proofs of Proposition \ref{general} and  Theorem \ref{WhenBNCisneeded}. 
	\begin{proof}[Proof of Proposition \ref{general}]
		First, we observe that for every $K\in \N^+$,
		\begin{align*}
			\lim\limits_{m\to \infty}\sup_{k\ge m,l\ge m}s_{k,k+l}
			&\le\lim\limits_{m\to \infty}\sup_{l\ge m}\sup_{k\ge K}s_{k,k+l}    
			=\limsup_{l\to\infty}\sup_{k\ge K} s_{k,k+l}.
		\end{align*}
		
		Next, for each $t>\lim\limits_{m\to \infty}\sup\limits_{k\ge m,l\ge m}s_{k,k+l}$,  there exists $m_0>K$ such that $	s_{k,k+l}<t$
		for all $ k\ge m_0,l\ge m_0, $ then
		\begin{equation} \label{ineqsklt}
			\limsup_{l\to\infty}\sup_{k\ge m_0} s_{k,k+l} \leq t.
		\end{equation}
	We proceed by a similar argument as in the proof of Lemma \ref{t=t^*}. For all  $k\in \N^+$ and $\eps>0$, it follows that
		\begin{align*}
			\prod_{i=k}^{k+l}\sum_{j=1}^{n_i}c_{i,j}^{s_{k,k+l}+\eps}
			&\le \om_{k,k+l}^\eps\to 0\ (l\to\infty)\\ 
			\prod_{i=k}^{k+l}\sum_{j=1}^{n_i}c_{i,j}^{s_{k,k+l}-\eps} 
			&\ge \om_{k,k+l}^{-\eps}\to +\infty\ (l\to\infty).
		\end{align*}
		Since 
		\begin{align*}
			\prod_{i=m_0}^{k'+l}\sum_{j=1}^{n_i}c_{i,j}^{s_{k',k'+l}}
			&=\frac1{\prod_{i=k'}^{m_0}\sum_{j=1}^{n_i}c_{i,j}^{s_{k',k'+l}}}\in \Big(\frac1{\prod_{i=k'}^{m_0}n_i},\,\frac1{\um_{k',m_0}^{d}\prod_{i=k'}^{m_0}n_i}\Big)
		\end{align*}
		for any $K\le k'<m_0$, we have 
		\[
		\lim_{l\to\infty} |s_{k',k'+l}-s_{m_0,k'+l}|=0,\qquad K\le k'<m_0. 
		\]	
		Similarly, for $K\le k'<m_0$, we also have $	\lim_{l\to \infty}|s_{m_0,m_0+l}-s_{m_0,k'+l}|=0$. These imply that 
		\begin{align*}
			&\limsup_{l\to\infty}\sup_{k\ge K} s_{k,k+l} \le \limsup_{l\to\infty} \big(\sup_{k\ge K}s_{k,k+l}-\sup_{k\ge m_0}s_{k,k+l}\big)+\limsup_{l\to\infty}\sup_{k\ge m_0}s_{k,k+l}\\
			&\le\limsup_{l\to\infty}\big(\max_{K\le k'<m_0}\left|s_{k',k'+l}-s_{m_0,k'+l}\right|+\max_{K\le k'<m_0}\left|s_{m_0,k'+l}-s_{m_0,m_0+l}\right|\big)  \\
			&\hspace{7cm}+\limsup_{l\to\infty}\sup_{k\ge m_0}s_{k,k+l}\\
			&=\limsup_{l\to\infty}\sup_{k\ge m_0}s_{k,k+l}
		\end{align*}
		By \eqref{ineqsklt}, we have $\limsup_{l\to\infty}\sup_{k\ge K} s_{k,k+l}\leq t$, and the conclusion follows by the arbitrariness of the choice of $t$.
	\end{proof}

	\begin{proof}[Proof of Theorem \ref{WhenBNCisneeded}]
		Since BBC holds, it follows from \eqref{dimHBAQ}  and Corollary \ref{cor_AD} that $\limsup_{l\to\infty}\sup_{k\ge 1} s_{k+1,k+l}=	\dima E\geq \dimqa E$.
		Then it suffices to prove
		\[
		\dimqa E\ge\limsup_{l\to\infty}\sup_{k\ge 1} s_{k+1,k+l}.
		\]
		
		By Theorem \ref{dayu}, we have $\dimqa E\ge t_*$.  
		By Proposition \ref{BC} and the fact that $\liminf_{k\to\infty}\f{\log \uc_k}{\log \om_k}>0$, there exist $K$ and $L$ such that for any $0<\eta<\frac12\liminf_{k\to\infty}\f{\log \uc_k}{\log \om_k}$,
		\[
		\f{\log \om_{k+1,k+l}}{\log \om_{k}}> \f{\log \uc_{k+1}}{\log \om_{k+1}}>\eta,\quad k>K,\ l>L.
		\]
		Therefore, by Corollary \ref{contr} and Proposition \ref{general}, we have
		\begin{equation*}
			\begin{aligned}
				\dimqa E\ge t_*\ge\limsup_{l\to\infty}\sup_{k>K}s_{k+1,k+l}=\limsup_{l\to\infty}\sup_{k\ge 1} s_{k+1,k+l}.
			\end{aligned}\qedhere
		\end{equation*}

	\end{proof}

	Finally, we discuss Example  \ref{1-2-3+}, \ref{BCandBNC},   \ref{exa'} and   \ref{exa} in details.
	\begin{proof}[Proof of Example \ref{1-2-3+}]
			(1) For $\mathcal{M} (J,\{2,2,\dots\}, \{ (2^{-2^k},2^{-2^k} ) \}_{k\ge 1} )$, it is clear that  it does not satisfy \eqref{club}. By Proposition \ref{BC}, the homogeneity of the structure together with $\sup_{k\ge 1} n_k=2<\infty$ ensures that BBC holds.
			
			(2) For $ \mathcal{M}\big(J,\{k+1\}_{k\ge 1},\big\{\big(\f1{2k},\ldots,\f1{2k}\big)\big\}_{k\ge 1}\big)$, it is clear that  it  satisfies \eqref{club}. Since $\sup_{k\ge 1} n_k=+\infty$, it follows from Proposition \ref{BC} that   BBC does not hold.
	\end{proof}

	\begin{proof}[Proof of Example \ref{BCandBNC}]
	Observe that
		\[
		\prod_{k=\f{n(n+1)}2}^{\f{(n+1)(n+2)}2-1}c_{k,2}=\f{1}{2^n+1}=c_{\f{n(n+1)}2,1}\,,\quad n\in \N^+.
		\]
		Therefore,
$\sup_{n\ge 1}\min  \{l:\om_{\f{n(n+1)}2,\f{n(n+1)}2+l}\le\uc_{\f{n(n+1)}2}\}=\sup_{n\ge 1} n=\infty.$ By Proposition \ref{BC}, BBC does not hold.
		
		To show that the structure satisfies the BNC, it suffices to verify that
		\begin{equation}\label{1/4BNC}
			\limsup_{\de\to 0}\sup_{x\in [0,1]}\# \Big\{\bu\in D(\de):J_{\bu}\cap B\big(x,\f\de4\big)\neq \emptyset \Big\}<\infty.
		\end{equation}
		Given $\de\in (0,1)$, we partition $D(\de)$ into three subsets:
		$$D(\de)=D_1(\de)\cup D_2(\de)\cup D_3(\de),$$ where
		\begin{gather*}
			D_1(\de)=\big\{\bu\in D^*:c_\bu\le \de<c_{\bu^-}=2c_\bu\big\},\\
			D_2(\de)=\big\{\bu\in D^*:c_\bu\le \de<c_{\bu^-}=(1+2^{-m})c_\bu\text{\ for some\ }m\in \N^+\big\},\\
			D_3(\de)=\big\{\bu\in D^*:c_\bu\le \de<c_{\bu^-}=(1+2^{n})c_\bu\text{\ for some\ }n\in \N^+\big\}.
		\end{gather*}
		Clearly,
		\begin{gather*}
			\sup_{x\in [0,1]}\#\Big\{\bu\in D_1(\de):J_\bu\cap B\big(x,\f\de4\big)\neq \emptyset\Big\}\le 2,\\
			\sup_{x\in [0,1]}\#\Big\{\bu\in D_2(\de):J_\bu\cap B\big(x,\f\de4\big)\neq \emptyset\Big\}\le 2.
		\end{gather*}
		
		By construction, for any fixed $\bu_0\in D_3(\de)$ with $c_{\bu_0^-} = (1 + 2^{n})c_{\bu_0}$ for some $n\ge 2$, the only word $\bu\in D_3(\de)$ satisfying $J_{\bu}\subseteq J_{\bu_0^-}$ is $\bu=\bu_0$. Hence, for any $J_\bu,\bu\in D_3(\de)$ with $J_\bu\neq J_{\bu_0}$, we obtain that
		\[
		d(J_{\bu},J_{\bu_0})\ge \min\{c_{\bu^-}-c_{\bu},c_{\bu_0^-}-c_{\bu_0}\}>\min\Big\{\f12c_{\bu^-},\f12c_{\bu_0^-}\Big\}>\f\de2.
		\]
		It follows that
		\[
		\sup_{x\in [0,1]}\#\big\{\bu\in D_3(\de):J_\bu\cap B\big(x,\f\de4\big)\neq \emptyset\big\}\le 1
		\] 
		Therefore, \eqref{1/4BNC} holds with constant 5.
	\end{proof}

	\begin{proof}[Proof of Example \ref{exa'}]
		Since $c_*>0$, (i) (a) in Proposition \ref{quasihomo}  implies that the structure is normal, and by \eqref{Eq_HBAdim}, one verifies that  
		$$	
		\dimh E=-\f{\log p}{\log 2},\quad\odimb E=-\f{\log q}{\log 2} ,  \quad \dima E=1,
		$$
		where $p,q$ satisfy	$p+p^3=(q+q^2)(q+q^3)=1.$

		It is easy to verify that $\uk_{l,\eta} =\{k:\frac {l}{k}>\eta \},$ and 	it follows from Theorem \ref{no.1} that
		\[
		\dimqa E=t_*=\lim_{\eta\to 0}\limsup_{l\to\infty}\sup_{k\in \uk_{l,\eta}} s_{k+1,k+l}=\lim_{\eta\to 0}\limsup_{l\to\infty}\sup_{k< \f l\eta} s_{k+1,k+l}.
		\]
		For each $n\geq 1$, define  $s(n)= s_{a_n+n,\,2a_n+n-1}$. Then 
		\[
		\big( {2^{-s(n)}}+ {4^{-s(n)}}\big)^{a_n-n}\big( {2^{-s(n)}}+ {2^{-s(n)}}\big)^n=1.
		\]
		Since $\lim_{n\to\infty}\f {n}{a_n}=0$,  $\lim_{n\to\infty } s(n)= -\f{\log r}{\log 2}  $, where $r+r^2=1.$
		Fix $\eta\in (0,1)$. For sufficently large $n$, we have $  s(n)=		\sup_{k<\f{a_n}\eta}s_{k+1,k+a_n} . $

		Given $k,l$ with $k<\f l{\eta}$, if there exist some intervals $[2a_i,2a_i+i)$ such that $[2a_i,2a_i+i)\cap[k+1,k+l]\neq\emptyset$, then $k+1-i<2a_i\le k+l$, which gives
		\[
		\lim_{l\to\infty}\sup_{[2a_i,2a_i+i)\cap[k+1,k+l]\neq\emptyset}\f{i}{l}=0.
		\]
		Hence, $\limsup_{l\to\infty}\sup_{k< \f l\eta} s_{k+1,k+l}\le\lim_{n\to\infty } s(n)= -\f{\log r}{\log 2}$. 
		Otherwise for every $i\in \N^+$, $[2a_i,2a_i+i)\cap[k+1,k+l]=\emptyset$, which implies  $\limsup_{l\to\infty}\sup_{k< \f l\eta} s_{k+1,k+l}\le -\f{\log r}{\log 2}$ as well. We combine these together and then obtain $\dimqa E=-\f{\log r}{\log 2}. $
	\end{proof}
	
	Finally, we establish Example \ref{exa} with the aid of the following well-known inequalities.
	\begin{lemma}\label{ine}
		\begin{enumerate}[(i)]
			\item $(1-x)^a>1-ax-ax^2,\quad a>0,\ x\in \left(0,\f12\right]$;
			\item $x^{-\f{1 }{x^\alpha-1}}\ge 2^{-\f{1 }{2^\alpha-1}},\quad \alpha>1,\ x\ge 2;$ 
			\item $\f\alpha\beta+\f{\log\beta-\log\alpha}{\beta\log 3}<1,\quad \beta>\alpha>1$.
		\end{enumerate}
	\end{lemma}

	\begin{proof}[Proof of Example \ref{exa}]
		Since  $
		\big(1-(k+1)^{-\alpha} \big)^{\f\alpha\beta}+ (k+1)^{-\alpha}>1 $
		and
		\begin{align*}
			\big(1- (k+1)^{-\alpha}\big)^{\f\alpha\beta+\f{\log\beta-\log\alpha}{\beta\log(k+1)}}+ (k+1)^{-\alpha -\f{\log\beta-\log\alpha}{\log(k+1)}} 
			<&\big(1- (k+1)^{-\alpha}\big)^{\f\alpha\beta}+\frac{\alpha}{\beta}(k+1)^{-\alpha}< 1,
		\end{align*}
		it follows that for every $k$, $	\f\alpha\beta<s_{k}<\min\ \{\f\alpha\beta+\f{\log\beta-\log\alpha}{\beta\log(k+1)},1 \} $.
		Then,  for all  $k,l$, 
		\begin{equation}\label{est}
			\f\alpha\beta<\min_{k\le i\le l}s_i\le s_{k,k+l}\le \max_{k\le i\le l}s_i<\min\Big\{\f\alpha\beta+\f{\log\beta-\log\alpha}{\beta\log(k+1)},1\Big\},
		\end{equation}
		and it implies  
		\begin{equation}\label{dyl}
			\lim_{m\to\infty}\sup_{k\ge m,l\ge m}s_{k,k+l}=\f\alpha\beta.
		\end{equation}
		
		It is clear that
		$$
		\lim_{l\to \infty}\sup_{k\ge 1}\om_{k,k+l}=\lim_{l\to \infty}\sup_{k\ge 1}\prod_{i=k}^{k+l}\Big(1-(i+1)^{-k}\Big) =1,
		$$
		\iffalse
		\begin{align*}
			\limsup_{l\to \infty}\sup_k\om_{k,k+l}&=\limsup_{l\to \infty}\sup_k\prod_{i=k}^{k+l}\left(1-\f1{(i+1)^k}\right)\\
			&=\limsup_{l\to \infty}\lim_{k\to\infty}\prod_{i=k}^{k+l}\left(1-\f1{(i+1)^k}\right)\\
			&=1
		\end{align*}
		\fi
		and
		\begin{equation}\label{con}
			\lim_{k\to\infty}\f{\log\uc_k}{\log \om_k}=-\lim_{k\to\infty}\f{\beta\log(k+1)}{\sum_{i=1}^k\log\left(1-(i+1)^{-\alpha}\right)}=\left\{\begin{array}{ll}
				+\infty,\quad & \alpha>1  \\
				\beta/\alpha, &\alpha=1  \\
				0, & \alpha<1
			\end{array}
			\right..
		\end{equation}
		By (ii) (a) in Proposition \ref{quasihomo},  the structure is not quasi-normal if  $\alpha\ge 1$. 
		
		Next, we   show that the structure is   not quasi-normal for $0<\alpha<1$.
		Fix $\de>0$. Observe that $f(x)=(x+1)^{\alpha-1}\log(x+1)$ is strictly decreasing on $[{\text{e}}^{\f 1{1-\alpha}}-1,+\infty)$ with $\lim_{x\to +\infty}f(x)=0$. For each sufficiently large $l$, there exists a unique integer $k(l)>{\text{e}}^{\f 1{1-\alpha}}-1$ such that
		\begin{equation}\label{non4}
			f(k(l)) \ge \f\de l>f(k(l)+1).
		\end{equation}
		Since $f\big(l^{\f 1{1-\alpha}}-1\big)=\f1{1-\alpha}\f{\log l}l >\f\de l $ for $ l>{\text{e}}^\de$,
		we have that
		\begin{equation}\label{non1}
			k(l)>	l^{\f1{1-\alpha}}-2.
		\end{equation}
		
		\iffalse
		\begin{equation}\label{non4}
			\f{\log(k(l)+1)}{(k(l)+1)^{1-\alpha}}\ge \f\de l>\f{\log(k(l)+2)}{(k(l)+2)^{1-\alpha}}.
		\end{equation}
		Note that $ l>{\text{e}}^\de$ and 
		\[
		{\left(l^{\f1{1-\alpha}}\right)^{\alpha-1}}		 {\log\left(l^{\f1{1-\alpha}}\right)}=\f1{1-\alpha}\f{\log l}l>\f{\log l}l>\f\de l,.
		\]
		\blue{	Thus 
			\begin{equation}\label{non1}
				k(l)>	l^{\f1{1-\alpha}}-2
			\end{equation}
			for any sufficiently large $l$. 
		}
		\fi
		
		Let $\varphi:(0,+\infty)\to (0,+\infty)$ be strictly increasing.   
		Since $0<\alpha<1$, for sufficiently large $k$, there exist   constants $C_\alpha,C'_\alpha>0$ depending only on $\alpha$, such that
		\begin{align*} 
			C_\alpha^{-1}(k+2)^{1-\alpha}<-\sum_{i=1}^k \log (1-(i+1)^{-\alpha})<C_\alpha(k+1)^{1-\alpha}
		\end{align*}
		and $	-\f1{\log(1-k^{-\alpha})}>C'_\alpha k^\alpha. $  Then it follows that  
		\begin{align*}
			&\varphi\Big(\frac{\log \um_{k(l)+1,k(l)+l}}{\log \om_{k(l)}}\Big)\bigg/\frac{\log \om_{k(l)+1,k(l)+l}}{\log \om_{k(l)}}\\
			=\ &\frac{\sum_{i=1}^{k(l)}\log (1-(i+1)^{-\alpha})}{\sum_{i=k(l)+1}^{k(l)+l} \log(1-(i+1)^{-\alpha})}\varphi\Big(-\frac{\beta\sum_{i=k(l)+1}^{k(l)+l}\log (i+1)}{\sum_{i=1}^{k(l)} \log(1-(i+1)^{-\alpha})}\Big)\\
			%>\ &\f{\sum_{i=1}^{k(l)}\log (1-(i+1)^{-\alpha})}{l\log(1-(k(l)+2)^{-\alpha})}\varphi\left(-\frac{\beta l\log (k(l)+1)}{\sum_{i=1}^{k(l)} \log(1-(i+1)^{-\alpha})}\right)\\
			%\ge &\frac{\sum_{i=1}^{k(l)}\log (1-(i+1)^{-\alpha})}{\sum_{i=k(l)+1}^{k(l)+l} \log(1-(i+1)^{-\alpha})}\varphi\left(-\frac{\beta\sum_{i=k(l)+1}^{k(l)+l}\log (i+1)}{\sum_{i=1}^{k(l)} \log(1-(i+1)^{-\alpha})}\right)\\
			>\ &    -\f{ (k(l)+2)^{1-\alpha}}{C_{\alpha}l\log(1-(k(l)+2)^{-\alpha})}\varphi\left(\frac{\beta l\log (k(l)+1)}{C_{\alpha} (k(l)+1)^{1-\alpha}}\right)          .
		\end{align*}
		Combining these  with   \eqref{non4} and \eqref{non1}, for sufficiently large $l$,  we have that
		\begin{align*}
			\varphi\Big(\frac{\log \um_{k(l)+1,k(l)+l}}{\log \om_{k(l)}}\Big)\bigg/\frac{\log \om_{k(l)+1,k(l)+l}}{\log \om_{k(l)}}>\ &\f{C'_\alpha(k(l)+2)}{C_\alpha l}\varphi\Big(\frac{\beta\de}{C_\alpha}\Big)\\
			>\ & C'_\alpha C_\alpha^{-1}l^{\f{\alpha}{1-\alpha}}\varphi\Big(\frac{\beta\de}{C_\alpha}\Big) 
		\end{align*}
		Since $0<\alpha<1$ and  $\lim_{l\to\infty}l^{\f{\alpha}{1-\alpha}}=\infty$, the structure is not quasi-normal.

		(i)	For $\alpha>1$, it is  clear  that $\lim_{k\to\infty}\om_k>0$. Hence, $\dimqa E=\dima E$=1. By Lemma \ref{ine} (iii) and \eqref{est}, we have
		\[
		\limsup_{l\to\infty}\sup_{k\ge 1} s_{k,k+l}\le\sup_{k,l}s_{k,k+l}\le \max\Big\{\f\alpha\beta+\f{\log\beta-\log\alpha}{\beta\log 3},s_{1}\Big\}<1.
		\]
		
		For each integer $k\geq 1$, we write
		\[
		h_k(t)=\prod_{i=1}^k\Big[\big(1- (i+1)^{-\alpha}\big)^t+(i+1)^{-\beta t}\Big],\quad t\in [\f\alpha\beta,1 ].
		\] 
		Since  $\alpha>1$, we have that  
		\begin{align*}
			h_k\big(\f\alpha\beta\big)
			&=\prod_{i=1}^k\Big[1+\big(1-\f\alpha\beta\big)(i+1)^{-\alpha}+O\big({(i+1)^{-2\alpha}}\big)\Big]
		\end{align*}
		converges.  Moreover $h_k(t)$ converges uniformly to the  continuous function 
		\[
		h_\infty(t)=\prod_{i=1}^\infty\Big[\big(1-(i+1)^{-\alpha}\big)^t+ (i+1)^{-\beta t}\Big],\quad t\in  [\f\alpha\beta,1].
		\]
		Since $h_\infty(t)$ is strictly decreasing with $h_\infty\big(\f\alpha\beta\big)>1>h_\infty(1),$
		 there exists a unique number $t_0\in (\f\alpha\beta,1 )$ such that $h_\infty(t_0)=1$. It follows from $h_k(s_{1,k})\equiv 1$ and the uniform convergence that
		\[
		\frac\alpha\beta<\lim_{k\to\infty} s_{1,k}=t_0<1.
		\]
		
		Furthermore, assume  $\beta\ge\f{2^\alpha+1}{2^\alpha-1}2^{\f{\alpha}{2^\alpha-1}}\alpha$.   We first show that $\{s_{k}\}_{k=1}^\infty$ is strictly decreasing.  
		Let
		\begin{equation}\label{def_Fxs}
			F(x,s)=\big(1- x^{-\alpha}\big)^{s}+x^{-\beta s}-1, \quad x\ge2,\ s>0.
		\end{equation}
		Since $F_s'<0$ for all  $x\ge2,s>0$, it follows from the implicit function theorem that there exists  a unique   $s(x)$ such that $F(x,s(x))=0$ for $x\ge 2$. By Lemma \ref{ine} (i) and (ii), we have
		\begin{align*}
			F\Big(x,\f{\alpha x^\alpha}{\beta(x^\alpha-1)}\Big)&>1-\f{\alpha}{\beta(x^\alpha-1)}-\f{\alpha}{\beta x^\alpha(x^\alpha-1)}+x^{-\alpha x^\alpha(x^\alpha-1)^{-1}}-1\\
			&\ge x^{-\alpha}\Big[2^{-\f{\alpha }{2^\alpha-1}}-\f{ \alpha(2^{\alpha}+1)}{\beta(2^\alpha-1)}\Big] \ge 0, 
		\end{align*}
		and it implies that $s(x)>\f{\alpha x^\alpha}{\beta(x^\alpha-1)}$. Furthermore,
		\begin{eqnarray}
			s'(x)&=&-\f {F'_x}{F'_s}\notag=\f{\alpha s(x)x^{-(\alpha +1)}\left(1-x^{-\alpha}\right)^{s(x)-1}-\beta s(x)x^{-(\beta s(x)+1)}}{\beta x^{-\beta s(x)}\log x-\left(1-x^{-\alpha}\right)^{s(x)}\log(1-x^{-\alpha})}\notag\\
			&=&\f{\beta s(x)x^{-(\alpha +1)}\left(1-x^{-\alpha}\right)^{s(x)-1}}{\beta x^{-\beta s(x)}\log x-\left(1-x^{-\alpha}\right)^{s(x)}\log(1-x^{-\alpha})}\Big[\f {\alpha}\beta-x^{\alpha-\beta s(x)}\left(1-x^{-\alpha}\right)^{1-s(x)}\Big].\label{implict}
		\end{eqnarray}
		Since  $F(x,s(x))=0$, that is, $x^{ -\beta s(x)} =  1- (1- x^{-\alpha} )^{s(x)}  $, by $s(x)>\f{\alpha x^\alpha}{\beta(x^\alpha-1)}$, we obtain
		\begin{equation*} %\label{mono}
			\begin{aligned}
				x^{\alpha-\beta s(x)}\big(1-x^{-\alpha}\big)^{1-s(x)}
				&=(x^\alpha-1)\big[\big(1-x^{-\alpha}\big)^{-s(x)}-1\big] >x^{-\alpha}(x^\alpha-1)s(x) >\f{\alpha}{\beta}.
			\end{aligned}
		\end{equation*}
		Combining it with \eqref{implict} , we have that   $s(x)$ is strictly decreasing for $x\ge 2$. Therefore, $\{s_{k}\}$ is strictly decreasing,  and it follows that  for each $l$,
		\[
		s_{k,k+l}\le \left\{\begin{array}{ll}
			s_{k,l+1}\le s_{1,l+1},\quad &1\le k\le l\vspace{0.5em}\\
			s_{l+1}<s_{1,l+1}, & k\ge l+1
		\end{array}\right..
		\]
		Hence, we obtain that $\limsup_{l\to\infty}\sup_{k\ge 1} s_{k,k+l}=\lim_{l\to \infty} s_{1,l}.$
		
		(ii) For $\alpha=1$, it is clear that $\lim_{k\to\infty}\om_k=0$. By Proposition \ref{general}, \eqref{est} and \eqref{dyl},
		\begin{align*}
			\f\alpha\beta\le\liminf_{l\to\infty}s_{1,l}\le\limsup_{l\to\infty}s_{1,l}\le \limsup_{l\to\infty}\sup_{k\ge 1}s_{k,k+l}=\lim_{m\to\infty}\sup_{k\ge m,l\ge m}s_{k,k+l}=\f\alpha\beta.
		\end{align*}
		Thus $\lim_{l\to\infty}s_{1,l}=\f\alpha\beta$.
		Observe that $S=\left\{\f 1n\right\}_{n\ge2}\subset E$ and by \cite[Theorem 3.4.7]{fraser2020assouad}, we have  $\dimqa S=1$. Therefore,
		\[
		\f\alpha\beta=\lim_{l\to\infty} s_{1,l}=\lim_{l\to\infty}\sup_{k\ge 1} s_{k,k+l}<\dimqa E=\dima E=1.
		\]
		
		(iii) For $\alpha<1$,  by \eqref{con}, Theorem \ref{dayu} and Theorem \ref{xiaoyu}, it follows that 
		$$
		\lim_{\eta\to 0}\limsup_{l\to\infty}\sup_{k\in \uk_{l,\eta}} s_{k+1,k+l}\\
		\le \dimqa E  \le\lim_{\eta\to 0}\limsup_{l\to\infty}\sup_{k\in \kk_{l,\eta}} s_{k,k+l}
		$$
		Similar to (ii), we have 
		\[
		\lim_{l\to\infty}s_{1,l}=\limsup_{l\to\infty}\sup_{k\ge 1} s_{k,k+l}=\lim_{m\to\infty}\sup_{k\ge m,l\ge m}s_{k,k+l}=\f\alpha\beta,
		\]
		and it implies that 
		$$
		\f\alpha\beta=\lim_{l\to\infty}s_{1,l}  \le \dimqa E \le \limsup_{l\to\infty}\sup_{k\ge 1} s_{k,k+l}	=\f\alpha\beta.
		$$
		
		Since $P=\left\{\prod_{i=1}^n\left[1-(i+1)^{-\alpha}\right]\right\}_{n\ge 2}\subset E$, by \cite[Example 1.12]{lu2016quasi}, $\dima P=1$.  Hence, we  have
		\[
		\f\alpha\beta=\lim_{l\to\infty} s_{1,l}=\lim_{l\to\infty}\sup_{k\ge 1} s_{k,k+l}=\dimqa E<\dima E=1. \qedhere
		\]
	\end{proof}
	\begin{remark}
		Example \ref{exa} (i) shows that even in the trivial case $\lim_{k\to\infty} \om_k>0$, $\limsup_{l\to\infty}\sup_{k\ge 1} s_{k,k+l}$ can still take values in $(0,1)$.
	\end{remark}
	\iffalse
	\begin{theorem}\cite{xi2017assouad}
		Suppose sequences $\{c_k\},\{n_k\}$ satisify $c_*>0$ and 
		\[
		\limsup_{l\to\infty}\sup_{k\ge 1}\frac{\log n_{k+1}\cdots n_{k+l}}{-\log c_{k+1}\cdots c_{k+l}}=\liminf_{k\to\infty}\frac{\log n_1\cdots n_k}{-\log c_1\cdots c_k}.
		\]
		Any homogeneous Moran set $E\in \mathcal{M}(J,\{n_k\},\{c_{k}\})$ is quasi-Lipschitz Assouad-minimal, that is, 
		\[
		\dimqa E=\dima E.
		\]
	\end{theorem}
	\fi

	%\nocite{*}
	%\bibliographystyle{abbrv}
	%\bibliography{Assouad}
	
\end{document}